\documentclass[11pt]{amsart}
\usepackage{amssymb,latexsym,amsmath,amsthm,amsfonts, enumerate}

\usepackage{color}
\usepackage[all]{xy}
\usepackage{etex}

\usepackage{tikz}
\usetikzlibrary{calc,decorations.markings,arrows}
\tikzset{->-/.style={decoration={markings,mark=at position .5 with {\arrow{>}}},postaction={decorate}}}

\usetikzlibrary{positioning,shapes,shadows,arrows,snakes}

\usepackage[colorlinks=true,pagebackref,hyperindex]{hyperref}
\newcommand\myshade{85}
\colorlet{mylinkcolor}{violet}
\colorlet{mycitecolor}{red}
\colorlet{myurlcolor}{cyan}

\hypersetup{
  linkcolor  = mylinkcolor!\myshade!black,
  citecolor  = mycitecolor!\myshade!black,
  urlcolor   = myurlcolor!\myshade!black,
  colorlinks = true,
}

\numberwithin{equation}{section}

\newtheorem{theorem}{Theorem}[section]
\newtheorem{proposition}[theorem]{Proposition}

\newtheorem{corollary}[theorem]{Corollary}
\newtheorem{lemma}[theorem]{Lemma}

\theoremstyle{definition}
\newtheorem{remark}[theorem]{Remark}
\newtheorem{example}[theorem]{Example}
\newtheorem{definition}[theorem]{Definition}

\newcommand{\Aut}{\mathrm{Aut}}
\newcommand{\QAut}{\mathrm{QAut}}
\newcommand{\WAut}{\mathrm{WAut}}

\newcommand{\Lat}{\mathrm{Lat}}
\newcommand{\monoto}{\hookrightarrow}
\newcommand{\za}{\alpha}
\newcommand{\zb}{\beta}

\newcommand{\zS}{\Sigma}

\newcommand{\calf}{\mathcal{F}}
\newcommand{\cala}{\mathcal{A}}

\newcommand{\A}{\mathcal{A}}

\def\T{\mathbb{T}}
\def\XX{\mathcal{X}}
\def\F{\mathcal{F}}
\def\x{{\bf{x}}}

\def\y{{\bf{y}}}

\def\p{{\bf{p}}}
\def\P{\mathbb{P}}
\def\S{\Sigma}

\newcommand{\Z}{\mathbb{Z}}
\newcommand{\Q}{\mathbb{Q}}

\subjclass[2010]{13F60}

\begin{document}
\title{Cluster automorphisms and quasi-automorphisms}
\author{Wen Chang}\thanks{The first author is supported by the NSF of China (Grant No. 11601295) and by Shaanxi Normal University.}
\address{School of Mathematics and Information Science, Shaanxi Normal University,
Xi'an 710062, China.}
\address{Department of Mathematics, University of Connecticut, Storrs, CT 06269-1009, USA}
\email{changwen161@163.com}
\author{Ralf Schiffler}\thanks{The second author was supported by the NSF-CAREER grant DMS-1254567 and by the University of Connecticut.}
\address{Department of Mathematics, University of Connecticut,
Storrs, CT 06269-1009, USA}
\email{schiffler@math.uconn.edu}

\begin{abstract} We study the relation between the cluster automorphisms and the quasi-automorphisms of a cluster algebra $\A$.
We proof that under some mild condition, satisfied for example by every skew-symmetric cluster algebra, the quasi-automorphism group of $\A$ is isomorphic to a subgroup of the cluster automorphism group of $\A_{triv}$, and the two groups are isomorphic if $\A$ has principal or universal coefficients;
 here  $\A_{triv}$ is the cluster algebra with trivial coefficients obtained from $\A$ by setting all frozen variables equal to the integer~1. 
 
 We also compute the quasi-automorphism group of all finite type and all skew-symmetric affine type cluster algebras, and show in which types it is isomorphic to the cluster automorphism group of $\A_{triv}$.
\end{abstract}

\maketitle
\setcounter{tocdepth}{1}

\tableofcontents

\section{Introduction}
 There are several notions of automorphisms of cluster algebras in the literature. For cluster algebras with trivial coefficients, cluster automorphisms were introduced in \cite{ASS12} as $\Z$-algebra homomorphisms that map a cluster to a cluster and commute with the mutations at that cluster. 
In the same paper, the authors gave several characterizations of cluster automorphisms and computed the cluster automorphism groups for finite and affine types.
This notion of automorphisms was later generalized to cluster algebras with arbitrary coefficients in several ways as follows.

 In \cite{ADS}, the authors defined cluster homomorphisms and constructed a category of (rooted) cluster algebras, and in \cite{CZ16,CZ16b}, the resulting cluster automorphism groups were studied. In this generalization, the cluster automorphisms are required to map the frozen variables to frozen variables, and this notion turned out to be too restrictive to include important symmetries of cluster algebras, notably the twist map on Grassmannians  of \cite{MS}.

Fraser introduced the less restrictive notions of quasi-homomorphisms and quasi-automorphisms in \cite{F16}. The definition of a quasi-automorphism  does not strictly require that a cluster is mapped to a cluster, but allows for the flexibility that the image of a cluster is a cluster up to rescaling each cluster variable by some coefficients. Fraser's definition is quite subtle, for example the set of all quasi-automorphisms does not form a group in general, see Example \ref{Qnotgroup}.
This led Fraser to construct the quasi-automorphism group $\QAut_0(\A)$ by considering certain equivalence classes, the proportionality classes, of quasi-automorphisms.

At this point the following questions are natural.
\begin{enumerate}
\item How does the quasi-automorphism group relate to the cluster automorphism group?
\item Can one compute the quasi-automorphism groups of cluster algebras of finite and affine types?
\end{enumerate}

In order to state our answers to these questions, it will be convenient to introduce some notation. Given a cluster algebra $\A$, we denote by $\A_{triv}$ the cluster algebra with trivial coefficients obtained from $\A$ by setting all frozen variables equal to the integer~1. We denote by $\QAut_0(\A)$ the quasi-automorphism group of $\A$ and by $\Aut^+(\A)$ the group of (direct) cluster automorphisms of $\A$. Let $\Aut^+_0(\A)$ be the quotient group of $\Aut^+(\A)$ by the subgroup of all automorphisms that fix each cluster variable, also known as 
 the subgroup of
coefficient permutations.

We say that a cluster algebra $\A$ satisfies condition $(\star)$ if $\A$ is skew-symmetric, or $\A$ is of finite type, or the exchange matrix of $\A$ is non-degenerate. Each of these conditions guarantees that the exchange graph of $\A$ only depends on the principal part of the exchange matrix $B$ and not on the coefficient system.

Our main result regarding question (1) is the following.

\begin{theorem}[Theorem \ref{thm:main-1}]
Let $\A$ be a cluster algebra. There is an isomorphism of groups
\[\QAut_0(\A)\cong \Aut^+(\A_{triv})\]
in each of the following cases.

{\rm (1)} $\A$ has principal coefficients and satisfies  condition $(\star)$.

{\rm (2)} $\A$ has universal coefficients.

{\rm (3)} $\A$ is a surface cluster algebra with boundary coefficients.
\end{theorem}

In general, without imposing any conditions on the coefficients of the cluster algebra $\A$, we have the following injective group homomorphisms
\[\Aut^+_0(\A)\monoto\QAut_0(\A)\stackrel{(\star)}{\monoto} \Aut^+(\A_{triv}),\]
where the first homomorphism always exists, see Proposition \ref{prop:main-1} and the second exists if $\A$ satisfies condition~$(\star)$, see Proposition \ref{prop:main-2}.
We give examples of finite type cluster algebras for which these injective homomorphisms are not surjective. Thus the conditions on the coefficients in the theorem above are necessary.

We then give an affirmative answer to question (2) and compute the quasi-automorphism groups of all cluster algebras of finite type as well as for all cluster algebras of skew-symmetric affine type, see Table \ref{table} below. \begin{table}[ht]
\begin{equation*}\begin{array}{ccc}
\begin{array}[t]{| lcl |}
\hline
\textrm{Dynkin type} &\quad& \QAut_0(\A) \\
\hline
\mathbb{A}_{n} (n\geqslant 2)&& \Z_{n+3} \\
\mathbb{A}_1 && \Z_2\\
\mathbb{B}_{n}  && \Z_{n+1} \\
\mathbb{C}_{n}  && \Z_{n+1} \\
\mathbb{D}_{4} && \Z_{4}\times G \\
\mathbb{D}_{2n} (n\geqslant 3)&& \Z_{2n}\times H \\
\mathbb{D}_{2n+1}&& \Z_{2n+1}\times \Z_2 \\
\mathbb{E}_6 && \Z_{14} \\
\mathbb{E}_7 && \Z_{10} \\
\mathbb{E}_8 && \Z_{16} \\
\mathbb{F}_4 && \Z_{7} \\
\mathbb{G}_2 && \Z_{4} \\\hline
\end{array}
&\qquad&
\begin{array}[t]{|lcl|}
\hline
\textrm{Affine type} &\quad& \QAut_0(\A) \\
\hline
&& \\
\widetilde{\mathbb{A}}_{p,q} (p\neq q)&& H_{p,q} \\[5pt]
\widetilde{\mathbb{A}}_{p,p}&& H_{p,p}\rtimes G_1 \\[5pt]
\widetilde{\mathbb{D}}_{4} && \Z\times G_2 \\[5pt]
\widetilde{\mathbb{D}}_{n-1} (n\geqslant 6)&& H \\[5pt]
\widetilde{\mathbb{E}}_6 && \Z\times G_3 \\[5pt]
\widetilde{\mathbb{E}}_7 && \Z\times G_1 \\[5pt]
\widetilde{\mathbb{E}}_8 && \Z \\[-2pt]
&&\\
\hline
\end{array}
\end{array}
\end{equation*}
\smallskip
\caption{The quasi-automorphism groups for cluster algebras of finite type (left), where $G$ is a subgroup of $S_3$ and $H$ is a subgroup of $\Z_2$,  and for cluster algebras of affine type (right), where $G_1$ is a subgroup of $\Z_2$, $G_2$ is a subgroup of $S_4$,  $G_3$ is a subgroup of $S_3$; the group $H_{p,q}$ is defined in equation (\ref{Hpq}), and the group $H$ satisfies $\Z\subseteq H \subseteq G$, where $G$ is given in equation (\ref{G}).   }
\label{table}
\end{table}

In particular, we show the following.
\begin{theorem} There is an isomorphism of groups
 \[\QAut_0(\A)\cong \Aut^+(\A_{triv})\]
in each of the following cases.

{\rm (1)} $\A$ is of finite type, but not type $\mathbb{D}_n$ with $n$ even.

{\rm (2)} $\A$ is of affine type $\widetilde{\mathbb{A}}_{p,q}$, with $p\ne q$, type $\widetilde{D}_{n-1}$, for all $n$, or type $\widetilde{\mathbb{E}}_8$.

{\rm (3)} $\A$ is of rank 2.
\end{theorem}

For completeness, we also mention yet another notion of automorphisms introduced by Saleh in \cite{Saleh} as automorphisms of the ambient field that restrict to a permutation of the set of all cluster variables. The relation between Saleh's notion and that of cluster automorphisms above was studied in \cite{ASS14} and led to the open question of unistructurality of cluster algebras, see also \cite{Bazier-Matte}.

The paper is organized as follows. After recalling some basic notions about cluster algebras in section \ref{Preliminaries}, we review the definitions of the different automorphisms in section~\ref{sect 3}. In that section, we also introduce  weak cluster automorphisms, which lie between cluster automorphisms and quasi-automorphisms. In section \ref{sect 4}, we study the relations between the different automorphism groups and prove Theorem \ref{thm:main-1}. Sections \ref{sect 5} and \ref{section:affine type} are devoted to the computation of the quasi-automorphism groups in finite and affine types, respectively, and section \ref{sect 7} contains the rank 2 case.

\section{Preliminaries on cluster algebras}\label{Preliminaries}

We recall basic definitions and properties on cluster algebras in this section. Throughout the paper, we use $\Z$ as the set of integers and use the notation $[x]_+=\max(x,0)$ for $x\in \Z$.
\subsection{Labeled seeds versus unlabeled seeds}
A \emph{labeled seed} of rank $n$ is a triple $\Sigma=(\x,\p,B)$, where
\begin{itemize}
\item  $\x=(x_{1}, \ldots  ,x_{n})$ is an ordered set with n elements;
\item  $\p=(x_{n+1},\ldots ,x_{m})$ is an ordered set with $m-n$ elements;

\item  $B=(b_{x_jx_i})\in M_{m\times n}(\Z)$ is a matrix labeled by $(\x\,\sqcup\,\p)\times \x$, and it is extended skew-symmetrizable, that is, there exists a diagonal matrix $D$ with positive integer entries such that $D\underline{B}$ is skew-symmetric, where $\underline{B}$ is a submatrix of $B$ consisting of the first n rows.
\end{itemize}

The ordered set $\x$ is the \emph{cluster} of the labeled seed $\Sigma$. The elements of $\x$ are called the \emph{cluster variables} of $\zS$ and the elements of $\p$ are called the  \emph{coefficient variables} of $\S$. We shall  write $b_{ji}$ for the element $b_{x_jx_i}$ in $B$ for brevity. The $n\times n$ matrix $\underline{B}$ is called the \emph{exchange matrix} of $\S$; its  rows are called \emph{exchangeable rows}, and the remaining rows of $B$ are called \emph{frozen rows}. We always assume that both $B$ and $\underline{B}$ are indecomposable matrices, and we also assume that $n>1$ for convenience.

Two labeled seeds $\Sigma=(\x, \p, B)$
and $\Sigma'=(\x', \p', B')$ are said to
define the same \emph{unlabeled seed}
if $\Sigma'$ is obtained from $\Sigma$ by simultaneous relabeling of the
ordered sets $\x$ and $\p$ and the corresponding relabeling of
the rows and columns of $B$.

In an unlabeled seed the cluster $\x$ and the coefficients $\p$ are sets that are not ordered. We shall use the notation $\x=\{x_1,\ldots,x_n\}, \p=\{x_{n+1},\ldots,x_m\} $ in an unlabeled seed and the notation $\x=(x_1,\ldots,x_n), $ $\p=(x_{n+1},\ldots,x_m) $ in a labeled seed.

For the remainder of the paper, our seeds will always be labeled seeds unless specified otherwise.

\subsection{Seed mutation}
Given an exchangeable cluster variable $x_k$, one may produce a new labeled seed by a seed mutation.

\begin{definition}
\label{def: mutation}
The labeled seed $\mu_k(\S)=(\mu_k(\x),\p,\mu_k(B))$ obtained by  \emph{mutation} of $\S$ in direction $k$ is given by:
\begin{itemize}
				\item $\mu_k(\x) = (\x \setminus \{x_k\}) \sqcup \{x'_k\}$ where
\begin{equation}
\label{eq:cluster-variable-mutation}
x_kx'_k = \prod_{\substack{1\leqslant j\leqslant m}} {x_j}^{[b_{jk}]_+} + \prod_{\substack{1\leqslant j\leqslant m}} {x_j}^{[-b_{jk}]_+}.
\end{equation}
				\item $\mu_k(B)=(b'_{ji})_{m\times n} \in M_{m\times n}(\Z)$ is given by
\begin{equation}
\label{eq:matrix-mutation}
b'_{ji} = \left\{\begin{array}{ll}
						- b_{ji} & \textrm{ if } i=k \textrm{ or } j=k~; \\
						b_{ji} + [-b_{jk}]_+b_{ki} + b_{jk}[b_{ki}]_+ & \textrm{ otherwise.}
					\end{array}\right.
\end{equation}
\end{itemize}
\end{definition}
It is easy to check that the mutation is an involution, that is $\mu_k\mu_k(\S)=\S$.
Note that in a seed mutation, the set $\p$ is not changed.

\subsection{Base ring $\Z\P$ and ambient field $\F$}\label{sect 2.3}
For a finite set $\p=\{x_{n+1},\ldots,x_{m}\}$, let $\P$ be the free abelian group (written
multiplicatively) generated by the elements of $\p$ and
define an addition~$\oplus$ in $\P$ by
\begin{equation}
\label{eq:tropical-addition}
\prod_j x_j^{a_j} \oplus \prod_j x_j^{b_j} =
\prod_j x_j^{\min (a_j, b_j)} \,.
\end{equation}
Then $(\P,\oplus,\cdot)$ is a semifield called  \emph{tropical semifield}. Let $\Z\P$ be its group ring and $\Q\P$ its field of quotients.
Let $\F=\Q \P(x_{1},x_{2},\ldots ,x_{n})$ be the field of rational functions in $n$ independent
variables with coefficients in $\Q \P$. $\F $ is called the \emph{ambient field} of the cluster algebra.

\subsection{Cluster algebra}
Recall that a \emph{tree} is a graph without cycles. An \emph{$n$-regular tree} $\T_n$ is a tree in which each vertex has precisely $n$ neighbors. We label the  edges of $\T_n$ by $1,\ldots,n$ in such a way that the $n$ edges emanating from each vertex receive different labels.

A \emph{cluster pattern} is an assignment of a labeled seed $\Sigma_t=(\x_t, \p_t, B_t)$ with rank $n$ to every vertex $t \in \T_n$, so that the seeds assigned to the endpoints of any edge labeled by $k$ are obtained from each other by the seed mutation in direction~$k$.
The elements of $\Sigma_t$ are written as follows:
\begin{equation}
\label{eq:seed-labeling}
\x_t = (x_{1;t}\,,\ldots,x_{n;t})\,,\quad
\p_t = (x_{n+1;t}\,,\ldots,x_{m;t})
\,,\quad
B_t = (b^t_{ij})\,.
\end{equation}
Since  the frozen variables  do not change under mutation, we have $\p_t=\p_{t'}$ for any $t, t' \in \T_n$.

Given a cluster pattern on $\T_n$ with initial labeled seed $\Sigma=(\x,\p,B)$, denote by
\begin{equation}
\label{eq:cluster-variables}
\XX
= \bigcup_{t \in \T_n} \x_t
= \{ x_{i,t}\,:\, t \in \T_n\,,\ 1\leq i\leq n \} \ ,
\end{equation}
 the set of all cluster variables in the cluster pattern. Note that the set $\XX$ does not change if we use unlabeled seeds instead of labeled seeds.

The \emph{cluster algebra} $\A$ associated with the cluster pattern (or associated with the seed $\S$) is the $\Z\P$-subalgebra of the ambient field $\F$
generated by $\XX$. Thus $\A = \Z\P[\XX]$.
 The elements $x_{i,t}\in \XX$ are called \emph{cluster variables},  $\P$ is the \emph{coefficient semifield} of $\A$, and  the elements in $\P$ are called \emph{coefficients} of $\A$.
The cluster algebra is called \emph{skew-symmetric} if its initial exchange matrix $\underline{B}$ is skew-symmetric.
 The cluster algebra is said to be of \emph{geometric type} if $\P$ is the tropical semifield defined in section \ref{sect 2.3}. In this paper, all cluster algebras are of geometric type.

\subsection{Iced valued quiver}
To the matrix $B$ in a seed $\Sigma=(\x,\p,B)$, one can associate an iced valued quiver $Q=Q(B)$, whose mutable vertices are labeled by cluster variables in $\x$,  whose frozen vertices are labeled by the  frozen variables in $\p$, and whose arrows and values are assigned by $B$ (see Example \ref{exm:quiver and matrix}, we refer to \cite{K12} for details). Then the principal part $\underline{B}$ of the matrix $B$ corresponds to the full subquiver  $\underline{Q}$ of $Q$ whose vertices are the mutable vertices.

\begin{example}\label{exm:quiver and matrix}
Let $B$ be the following matrix, where $\underline{B}$ is skew-symmetrizable with skew-symmetrizing diagonal matrix $D=\textup{diag}\{2,2,1,1\}$.

\[ \begin{array}{c}B=\left(
                           \begin{array}{c}
                             \underline{B} \\
                             B' \\
                           \end{array}
                         \right)
=\left(
\begin{array}{cccc}
                             0 & 1 & 0 & 0\\
                            -1 & 0 & -1 & 0\\
                            0 & 2 & 0 & 2\\
                            0 & 0 & -2 & 0\\
                            0 & 0 & 0 & -1\\
\end{array}
\right)
\end{array}\]
The quiver corresponding to $B$ is as follows, where we frame the frozen vertex.
\[ \begin{array}{c}
Q(B)= \vcenter{
\xymatrix{1\ar[r]&2&3\ar[l]_{(2,1)}\ar@<1.6pt>[r]\ar@<-1.6pt>[r]&4\ar[r]&\framebox{5}}}\\
\end{array}\]

\end{example}

\subsection{Special coefficient systems} In this subsection, we recall the definitions of several special choices for the frozen variables $\p$ and the frozen rows of $B$.
\subsubsection{Trivial coefficients}
A cluster algebra is said to have trivial coefficients if its initial seed is of the form $\zS=(\x,\p,B)$ with $\p=\emptyset$ and $B=\underline{B}$.

If $\A$ is an arbitrary cluster algebra with initial  seed $\zS=(\x,\p,B)$, we denote by   $\A_{triv}$ the cluster algebra defined by the initial seed $\S_{triv}=(\x,\underline{B})$. Thus $\A_{triv}$ is obtained from $\A$ by setting all frozen variables to 1. We call $\A_{triv}$ the \emph{principal part cluster algebra} of $\A$.

\subsubsection{Principal coefficients} A cluster algebra $\cala$ is said to have \emph{principal coefficients} if it has a seed $((x_1,\ldots,x_n),(x_{n+1},\ldots,x_{2n}),B)$ such that $B=\left(\begin{smallmatrix} \underline{B}\\I\end{smallmatrix}\right)$, where $\underline{B}$ is the principal part of  $B$ and $I$ is the $n\times n$ identity matrix.
Principal coefficients were introduced in \cite{FZ07}, where it is shown that from the Laurent expansion of a cluster variable with principal coefficients one can easily obtain the Laurent expansion of the `same'  cluster variable with arbitrary coefficients.

\subsubsection{Universal coefficients} We recall the definition of universal coefficients from \cite{FZ07}.

\begin{definition}
\label{defn:coeff-specialization}
Let $\A$ and $\overline \A$ be cluster algebras
of the same rank~$n$ over the coefficient semifields $\P$
and~$\overline \P$, respectively,
with the respective families of cluster variables $(x_{i;t})_{i \in
  [1,n], t \in \T_n}$ and
$(\overline x_{i;t})_{i \in [1,n], t \in \T_n}$.
We say that $\overline \A$ is obtained from $\A$ by a
\emph{coefficient specialization} if
\begin{enumerate}
\item[(1)] $\A$ and $\overline \A$ have the same exchange
matrices $ \underline{B_t} = \underline{\overline B_t}$
 at each vertex $t \in \T_n$;

\item[(2)] there is a homomorphism of multiplicative
groups $\varphi: \P \to \overline \P$ that extends to a (unique) ring
homomorphism $\varphi: \A \to \overline \A$ such that
$\varphi(x_{i;t}) = \overline x_{i;t}$ for all~$i$ and~$t$.
\end{enumerate}
\end{definition}

In particular, a coefficient specialization from $\cala$ to itself is a ring homomorphism $\varphi\colon\cala\to \cala$  that fixes every cluster variable and is an endomorphism of the multiplicative group $\P$.

\begin{definition}
\label{def:universal-coeffs}
A cluster algebra $\A$ has \emph{universal coefficients}
if every cluster algebra with the same family of exchange matrices $(\underline{B}_t)$
is obtained from $\A$ by a \emph{unique} coefficient specialization.
\end{definition}
Universal coefficients were constructed for the cluster algebras of finite type in \cite{FZ07} and for surface type in \cite{R14,R15}.
 We give an example of finite type $\mathbb{A}_2$ which is different from the one given in \cite{FZ07} in Example \ref{ex WAut}.

\subsubsection{Gluing free coefficients}
A seed $\Sigma=(\x,\p,B)$ is said to be \emph{gluing free}, if any two frozen rows of $B$ are different.  It is not hard to see that a mutation of a gluing free seed is still gluing free \cite{CZ16}. We say a cluster algebra is \emph{gluing free}, if its seeds are gluing free.

The gluing free condition is important in the rest of the paper when we study the cluster automorphism groups and the quasi-automorphism groups.

\subsection{Exchange graph}
The \emph{exchange graph} of a cluster algebra is the $n$-regular graph whose vertices are the unlabeled seeds of the cluster algebra and whose
edges connect the unlabeled seeds related by a single mutation.
The exchange graph of a cluster algebra is thus a quotient graph of the $n$-regular tree.
It is conjectured by Fomin-Zelevinsky in  \cite{FZ03} that the exchange graph does not depend on the frozen part of the matrix. This conjecture is  proved for several types of cluster algebras which we list in the following lemma.

\begin{lemma}
\label{lem:dependence-exchange-graph}
Let $\A$ be a cluster algebra with initial seed $\zS=(\x,\p,B)$.   Then  the exchange graph only depends on the exchange matrix $\underline{B}$
if one of the following three conditions holds.
\begin{itemize}
\item[($\star$1)] $\A$ is of finite type \cite{FZ03b};
\item[($\star$2)] $\A$ is skew-symmetric \cite{CKLP13};
\item[($\star$3)] $\underline{B}$ is non-degenerate \cite{GSV08}.
\end{itemize}
We shall say that the cluster algebra satisfies the  condition \emph{($\star$)} if one of the three conditions holds.
\end{lemma}

\section{Automorphisms, weak automorphisms and quasi-automorphisms}\label{sect 3}
In this section, we recall the notions of cluster automorphisms \cite{ASS12,CZ16} and quasi-automorphisms \cite{F16}, and we also introduce the notion of weak automorphisms. We study the relation between these notions in section~\ref{sect 4}.
\subsection{Quasi-homomorphisms}

In this subsection, we recall Fraser's definition of quasi-homomorphisms from \cite{F16}. We first need to recall the definition of $y$-variables from \cite{FZ07}.

\begin{definition}
\label{def:hatted variables}
Let $\zS=(\x,\p,B)$ be a seed with
\begin{equation}
\x = (x_1, \ldots, x_n),\quad
\p = (x_{n+1}, \ldots, x_{m}),\quad
B= (b_{ji}).
\end{equation}
Define
$\y = (y_1, \ldots, y_n)$ and $\hat\y = (\hat y_1, \ldots, \hat y_n)$ to be the  $n$-tuples
given by
\begin{equation}
\label{eq:yhat}
y_{i} = \prod_{n+1 \leq j\leq m} x_{j}^{b_{ji}} \qquad \textup{ and }\qquad \hat y_{i} = y_{i}\prod_{1 \leq j\leq n} x_{j}^{b_{ji}} .
\end{equation}
The elements in $\y$  are called the  \emph{$Y$-variables} and the elements in  $\hat \y$  the \emph{hatted Y-variables} of the seed $\zS$.
\end{definition}

Now we recall the proportionality relation between two seeds; we do not use the original definition but rather an equivalent condition, see  \cite[Proposition 2.3]{F16}.

\begin{definition}
(1) Two elements $x,y \in \F$ are \emph{proportional}, written $x \asymp y$, if  the quotient $\frac{x}{y}$ lies in $ \P$.

(2) Two clusters $\x_{t}$ and $\x_{t'}$  are  \emph{proportional}, written $\x_t\asymp\x_{t'}$,  if $x_{i;t} \asymp x_{i;t'}$ for all $1 \leq i\leq n$.

(3) Two labeled seeds $\S_{t} = (\x_t,\p_t,B_t)$ and $\S_{t'} = (\x_{t'},\p_{t'},B_{t'})$  are  \emph{proportional}, written $\zS_t\asymp\zS_{t'}$,  if for all $1 \leq i\leq n$
\[x_{i;t} \asymp x_{i;t'}, \qquad \hat{y}_{i;t} = \hat{y}_{i;t'} \qquad\textup{ and }\qquad  \underline{B_t} = \underline{B_{t'}}.\]
\end{definition}

\begin{lemma}\label{lem:proportional seed which are equal}
Let $\A$ be a cluster algebra that satisfies the $(\star)$ condition in Lemma \ref{lem:dependence-exchange-graph}, then

(1) two clusters $\x$ and $\x'$ are proportional if and only if they are equal;

(2) two seeds $\zS$ and $\zS'$ are proportional if and only if they are equal.
\end{lemma}
\begin{proof}
(1) If $\x$ and $\x'$ are proportional, we see that $\pi(\x)=\pi(\x')$ under the coefficient specialization $\pi$ from $\A$ to $\A_{triv}$, which maps all the coefficients to $1$. On the other hand, since the cluster algebras we consider are of geometric type, the clusters determine the seeds  \cite{GSV08}.
On the other hand,  by Lemma \ref{lem:dependence-exchange-graph} the algebras $\A$ and $\A_{triv}$  have the same exchange graph, thus
 we have $\x=\x'$.

(2) This follows from the first part of the lemma.
\end{proof}

We are ready for the main definition of this subsection. Again, we do not use the original definition but rather an equivalent condition, see  \cite[Proposition 3.2]{F16}.

\begin{definition}
\label{defn:quasi-homomorphism}
Let $\cala$ and $\overline{\cala}$ be two cluster algebras of the same rank $n$ with coefficient semifields $\P$ and $\overline{\P}$ and with ambient fields $\calf$ and $\overline{\calf}$, respectively. A \emph{quasi-homomorphism} is a field homomorphism $\Psi \colon \F \to \overline{\F}$ such that $\Psi(\P)\subset\overline{\P}$ and there exists a seed   $\zS=(\x,\p,B)$ of $\cala$ and a seed $\overline{\zS}=(\overline{\x},\overline{\p},\overline{B})$ of $\overline{\cala}$ such that
\begin{equation}\label{qhsimdefn}
\Psi(\S) \asymp \overline{\S},
\end{equation}
 where $\Psi(\S)=(\Psi(\x),\Psi(\p),B)$ is the triple obtained by evaluating $\Psi$ on $\S$.
\end{definition}

\begin{remark}
 It is shown in \cite{F16} that if $\Psi$ is a quasi-homomorphism then the condition that $\Psi(\zS')$ is proportional to a labeled seed of $\overline{\A}$ actually holds for \emph{every} seed $\zS'$ in $\A$.
\end{remark}
The following result gives a useful interpretation of quasi-homomorphisms in terms of matrices.
\begin{lemma}[{\cite[Corollary 4.5]{F16}}] \label{lemma:matrix-describe-QHom}
Let $\A$ and $\overline{\A}$ be two cluster algebras, and let $\S = ((x_{1},\ldots,x_{n}), (x_{n+1},\ldots,x_{m}), B)$ be a seed in $\A$ and $\overline{\S}=((\overline{x}_{1},\ldots,\overline{x}_{n})$, $(\overline{x}_{n+1},\ldots,\overline{x}_{\overline{m}})$,$ \overline{B})$ a seed in $\overline{\A}$.

(a) If $\psi\colon\cala\to\overline{\cala} $ is a quasi-homomorphism such that  $\psi(\zS)\asymp\overline{\zS}$  then there are integers $m_{ij}$ such that
\begin{equation}\label{eq:matrix-describe-QH-1}
\displaystyle \Psi(x_j) = \prod_{i=1}^{\overline{m}}\overline{x}_i^{\, m_{ij}},
\end{equation}
for $j=1,\ldots,m$. Moreover, the $\overline{m}\times m$ integer matrix $M_\psi=(m_{ij})$ satisfies
\begin{equation}\label{newbtilde}
\overline{B}= M_\Psi {B}.
\end{equation}

(b) There exists a quasi-homomorphism $\psi\colon\cala\to \overline{\cala}$ such that $\psi(\zS)\asymp\overline{\zS}$ if and only if  the principal parts of $B,\overline{B}$ agree,
and the integer row span of $B$ contains the integer row span of $\overline{B}$.
\end{lemma}

\subsection{Quasi-automorphisms}

In this subsection, we recall the definition and properties of the quasi-automorphism group from \cite{F16}.

\begin{definition}
Let $\A$ and $\overline{\A}$ be two cluster algebras.

(1) Two quasi-homomorphisms $\Psi_1,\Psi_2$ from $\A$ to $\overline{\A}$ are called \emph{proportional} if \begin{equation}\label{eq:proportionality}
\Psi_1(\Sigma) \asymp \Psi_2(\Sigma)
\end{equation}
for some (hence every) seed $\Sigma$ of $\A$.

(2) A quasi-homomorphism $\Psi$ from $\A$ to $\overline{\A}$ is said to be a \emph{quasi-isomorphism} if there is a quasi-homomorphism $\Phi$ from $\overline{\A}$ to $\A$ such that $\Phi \circ \Psi$ is proportional to the identity map on $\A$. In this case $\Psi$ and $\Phi$ are \emph{quasi-inverses} of one another.

(3) A quasi-automorphism of $\A$ is a quasi-isomorphism from $\A$ to itself.
Let  $\QAut(\A)$ denote the set of all quasi-automorphisms of $\cala$.

\end{definition}

The following example shows that  the set $\QAut(\A)$ is not a group.
\begin{example}\label{Qnotgroup}
Let $\cala $ be the cluster algebra of type $\mathbb{A}_2$ with initial seed $\zS=((x_1,x_2)$, $(x_3), B)$, where the matrix $B$ and its corresponding quiver are as follows
\[ B= \left(\begin{array}{cc} 0&1\\-1&0\\3&0\end{array}\right)\qquad Q(B)= \xymatrix {\framebox{3}\ar@<3.2pt>[r]\ar@<0pt>[r]\ar@<-3.2pt>[r]&1\ar[r]&2}\]
and where the framed vertex corresponds to the frozen variable.
  Mutating this seed in direction 1 and 2 produces the seed $\zS'=((x_1',x_2'),(x_3), B')$,
  where $x_1'=(x_2+x_3^3)/x_1$ and $x_2'=(x_2+x_3^3+x_1x_3^3)/x_1x_2$ and
  \[ B'=  \left(\begin{array}{cc} 0&1\\-1&0\\0&-3\end{array}\right) \qquad Q(B')=\xymatrix{1\ar[r]&2\ar@<3.2pt>[r]\ar@<0pt>[r]\ar@<-3.2pt>[r]&\framebox{3}}.\]
Define $\Psi$ by
\[\Psi(x_1)=x_1'x_3^{-3},\quad \Psi(x_2)=x_2'x_3^{6} \quad \textup{ and }\quad\Psi(x_3)=x_3^2.\]
Then $\Psi(x_1)/x_1'\in \P$ and $\Psi(x_2)/x_2'\in \P$, and
\[\Psi(\hat y_1,\hat y_2)=\Psi(x_2^{-1}x_3^3,x_1)=(x_2'^{-1},x_1'x_3^{-3})=(\widehat {y_1'},\widehat {y_2'});\]
so $\Psi(\zS)\asymp\zS'$ and thus $\Psi$ is a quasi-homomorphism.

Define  a quasi-inverse $\Phi$ by
\[ \Phi(x_1')=x_1x_3^6,\quad \Phi(x_2')=x_2x_3^{-3}\quad\textup{ and }\quad \Phi(x_3)=  x_3^2.\]
  Then $\Phi\Psi(x_1)/x_1, \Phi\Psi(x_2)/x_2\in \P $ and
\[\Phi\Psi(\hat y_1,\hat y_2)=\Phi\Psi(x_2^{-1}x_3^3,x_1)=\Phi({x_2'}^{-1}, x_1'x_3^{-3})=(x_2^{-1}x_3^3, x_1)=(\hat y_1,\hat y_2).\]
This shows that $\Phi\circ\Psi$ is proportional to the identity and thus $\Psi $ is a quasi-automorphism.

Note that $\Psi$ is not invertible in $\QAut(\A)$. Indeed, if $\Psi$ had an inverse $\Psi^{-1}$ then $x_3=\Psi^{-1}\Psi(x_3)=\Psi^{-1}(x_3^2)$, which would imply $\Psi^{-1}(x_3)=x_3^{1/2}$ but this is not an element of $\P$. Thus $\QAut(\A)$ is not a group.
\end{example}

It is shown in \cite{F16} that the set of proportionality classes of quasi-automorphisms does form a group under composition. This leads to the following definition.
\begin{definition}\label{defn:quasi-automorphism-groups}
The \emph{quasi-automorphism group} $\QAut_0\A$ is the group of proportionality classes of quasi-automorphisms of $\A$.
\end{definition}

\begin{remark}\label{rem:quasi-automorphism}
For a matrix $B$, let $\Lat(B)$ denote the integer lattice generated by the rows of $B$.
Let $\A$ be a cluster algebra, and $\S,\S'$ be seeds in $\cala$. Then by Lemma \ref{lemma:matrix-describe-QHom}, $\Psi$ is a quasi-automorphism of $\A$ with $\Psi(\S)\asymp \S'$ if and only if $\underline{B}=\underline{B'}$ and $\Lat(B)=\Lat(B')$.
\end{remark}

The following lemma will be useful later.
\begin{lemma}
 \label{lem one frozen row}
Let $B=\left(\begin{smallmatrix}\underline{B}\\ M\end{smallmatrix}\right)$ be a $(n+m)\times n$ matrix and $B'=\mu_\omega(B)=\left(\begin{smallmatrix}\underline{B'}\\ M'\end{smallmatrix}\right)$ be a matrix obtained from $B$ after a sequence of mutations $\omega$. If $\Lat(B)=\Lat(B')$ for every matrix $M$ with $m=1$, then $\Lat(B)=\Lat(B')$ for every matrix $M$ with $m\geqslant 1$.
\end{lemma}
\begin{proof} Suppose $m\ge 1$ and
let $\zb,\zb'$ be any rows in $M,M'$ respectively. Then our assumption implies that $\zb\in \Lat( B')$,  and $\zb'\in \Lat (B)$, and thus every row of $M$ lies in $\Lat(B')$ and every row of $M'$ lies in $\Lat(B)$. By our assumption, we also have every row of $\underline{B}$ lies in $\Lat(B')$ and every row of $\underline{B'}$ lies in $\Lat(B)$. Therefore $\Lat(B)=\Lat(B')$.
\end{proof}

\subsection{Cluster automorphisms and weak cluster automorphisms} In this subsection we recall the definition of cluster automorphisms.  For cluster algebras with trivial coefficients this definition is due to \cite{ASS12} and for cluster algebras with arbitrary coefficients to \cite{CZ16}. We also introduce the notion of weak cluster automorphisms.

\begin{definition}\label{defn:cluster-automorphisms}
Let $\cala $ be a cluster algebra.

(1) A $\Z$-algebra automorphism $f\colon\cala\to \cala$ is called a \emph{cluster automorphism} if there exists a seed $(\x,\p,B)$ such that
\begin{itemize}
\item [(\rm i)] $f(\x)$ is a cluster,
\item [(\rm ii)] $f(\p)=\p$ as unordered sets,
\item [(\rm iii)] $f$ commutes with mutations, that is, for every $x,x'\in \x$,
\begin{equation}\label{eq:cluster-automorphisms}
f(\mu_{x,\x}(x'))=\mu_{f(x),f(\x)}(f(x')).
\end{equation}
\end{itemize}
\smallskip

(2) A $\Z$-algebra automorphism $f\colon\cala\to \cala$ is called a \emph{weak cluster automorphism} if there exists a seed $(\x,\p,B)$ such that conditions (i) and (iii) above are satisfied and
\begin{itemize}
\item [(\rm ii')] $f(\p)\subset\P$.
\end{itemize}
\end{definition}
Thus a weak cluster automorphism is allowed to map a frozen variable $x_{n+i}$ to an arbitrary element of $\P$, while a cluster automorphism must map frozen variables to frozen variables. The following lemma is obvious.

\begin{lemma}
 Every cluster automorphism is a weak cluster automorphism.\qed
\end{lemma}
 We show that the converse of the lemma does not hold in Example \ref{ex WAut} below.

It is shown in \cite[Proposition 2.4]{ASS12} that the condition \eqref{eq:cluster-automorphisms} holds for one cluster if and only if it holds for every cluster.
In our case, the cluster algebra is of geometric type, so the clusters determine the seeds \cite{GSV08}. Denote by $B$ and $B'$
the matrices of $\x$ and $f(\x)$ respectively. We say $B\cong B'$, if $B'$ is obtained from $B$ by simultaneous relabeling of the exchangeable rows and corresponding columns and the relabeling of the frozen rows.

 If $f$ is a cluster automorphism then $B\cong B'$ or $B\cong -B'$, and if $f$ is a weak cluster automorphism then we have the same relation for the principal parts of the matrices, that is, $\underline{B}\cong \underline{B}'$ or $\underline{B}\cong -\underline{B}'$. For cluster automorphism with trivial coefficients, this is shown in   \cite[Lemma 2.3]{ASS12} and the same proof generalizes to the other cases.

A cluster automorphism such that $B \cong B'$ is called a  \emph{direct cluster automorphism} and a  cluster automorphism such that   $B \cong -B'$  is an \emph{inverse  cluster automorphism}.
Similarly, a weak cluster automorphism such that $\underline{B} \cong \underline{B} '$ is called a  \emph{direct weak cluster automorphism} and a weak cluster automorphism such that   $\underline{B}  \cong -\underline{B} '$  is an \emph{inverse weak cluster automorphism}.

\begin{definition}
(1) Let $\Aut(\A)$ be the group of all cluster automorphisms of $\cala$, and let $\Aut^+(\cala)$ be the subgroup of $\Aut(\A)$ of all direct cluster automorphisms.

(2) Let $\WAut(\A)$ be the group of all weak cluster automorphisms of $\cala$, and let $\WAut^+(\cala)$ be the subgroup of $\WAut(\A)$ of all direct cluster automorphisms.
 \end{definition}

If two frozen rows in a matrix of a cluster algebra $\A$ coincide, then exchanging the corresponding frozen variables induces a cluster automorphism of $\A$. Note that the cluster variables and the clusters are not changed under such a cluster automorphism.  Such an automorphism is called a \emph{coefficient permutation}.

Denote by $\Aut_0(\A)$, $\Aut_0^+(\A)$, $\WAut_0(\A)$ and $\WAut_0^+(\A)$ the quotient groups of $\Aut(\A)$, $\Aut^+(\A)$, $\WAut(\A)$ and $\WAut^+(\A)$ by the respective subgroup of all coefficient permutations. If $\cala$ is a gluing free cluster algebra, then $\Aut_0(\A)\cong \Aut(\A)$ and $\Aut_0^+(\A)\cong \Aut^+(\A)$.

\begin{remark}\label{rem:cluster-automorphisms}
By  definition, a cluster automorphism $f$ maps a cluster to a cluster,
and induces an automorphism of the exchange graph as well as an automorphism of the $n$-regular tree.
\end{remark}

\begin{example}[A week cluster automorphism that is not a cluster automorphism]\label{ex WAut}
Let $\A$ be the cluster algebra of type $\mathbb{A}_2$ given by the initial seed $((x_1,x_2),(x_3, x_4,x_5,x_6,x_7), B)$, where the matrix $B$ and its corresponding quiver are as follows
\[ \begin{array}{ccc}B=\left(
\begin{array}{cc}
 0&1 \\ -1&0\\2&1\\1&1\\-1&0\\0&-1\\1&-1
\end{array}
\right) &\qquad&Q(B)= \vcenter{
\xymatrix{&\framebox{3}\ar@<1.6pt>[d]\ar@<-1.6pt>[d]\ar[dr]&\framebox{4}\ar[d]\ar[dl]\\\framebox{5}&1\ar[l]\ar[r]
&2\ar[dl]\ar[d]\\&\framebox{7}\ar[u]&\framebox{6}}}\\
\end{array}\]
The  framed vertices in the quiver are frozen. Then the five cluster variables of $\A$ are $x_1,x_2,x_1',x_2',x_1''$, where

\begin{equation*}
\label{cluster variables}
x'_1=\frac{x_3^2x_4x_7 + x_5x_2}{x_1},\
x'_2=\frac{x_3x_4x_1 + x_6x_7}{x_2},\
x''_1=\frac{x_3^2x_4x_6x_7 + x_3^3x_4^2x_1 +x_5x_6 x_2}{x_1x_2}.
\end{equation*}
Mutation at $x_1$ produces the seed $((x_1',x_2),(x_3, x_4,x_5,x_6,x_7), B)$
with
\[ \begin{array}{ccc}B'=\left(
\begin{array}{cc}
 0&-1 \\ 1&0\\-2&3\\-1&2\\1&0\\0&-1\\-1&0
\end{array}
\right) &\qquad&Q(B')= \vcenter{
\xymatrix{&\framebox{3}\ar@<0pt>[dr]\ar@<3.2pt>[dr]\ar@<-3.2pt>[dr]&\framebox{4}\ar@<1.6pt>[d]\ar@<-1.6pt>[d]\\
\framebox{5}\ar[r]&1\ar[d]\ar@<1.6pt>[u]\ar@<-1.6pt>[u]\ar[ur]
&2\ar[l]\ar[d]\\&\framebox{7}&\framebox{6}}}\\
\end{array}\]
Define a $\mathbb{Z}$-algebra automorphism $\tau\colon\A\to\A$ by
\[\begin{array}{llllllll}
\label{map tau}
x_1\mapsto  x_2, &
x_2\mapsto  x'_1,&
x_3\mapsto  x_3x_4x_5^{-1}, &
x_4\mapsto  x_3^{-1}x_4^{-1}x_5^2,&
x_5\mapsto  x_6, &
x_6\mapsto  x_7,\\
x_7\mapsto  x_3^2x_4.
\end{array}\]
A direct computation shows that $\tau$ is the following permutation of the 5 cluster variables
\[
\label{WeakCA tau}
x_1\mapsto  x_2
\mapsto  x'_1 \mapsto x''_1 \mapsto  x'_2\mapsto x_1,
\]
 and it  is a weak cluster automorphism of order $5$. However $\tau$ is not a cluster automorphism since $B'\ncong\pm B$.

We define another $\mathbb{Z}$-algebra automorphism $\sigma\colon\A\to\A$ by
\[\begin{array}{llllll}
\label{map sigma}
x_1\mapsto  x_2,&
x_2\mapsto  x_1,&
x_3\mapsto  x_5^{-1}x_6,&
x_4\mapsto  x_5^{2}x_6^{-1},&
x_5\mapsto  x_3x_4,&
x_6\mapsto  x_3^2x_4,\\
x_7\mapsto  x_7 .
\end{array}\]
Another direct computation shows that $\sigma$ is the following permutation
\begin{equation*}
\label{WeakCA tau2}
x_1\leftrightarrow x_2, \quad
x'_1\leftrightarrow x'_2, \quad
x''_1\leftrightarrow x''_1,
\end{equation*}
 and it is a weak cluster automorphism of order $2$. Again, $\sigma$ is not a cluster automorphism.

Note that $\tau\sigma=\sigma\tau^{-1}$, so $ \langle\tau, \sigma\rangle$ is the dihedral group $ D_5$. On the other hand, we have $\WAut(\A)=\WAut_0(\A) \subseteq \Aut(\A_{triv})\cong D_5$, where the last isomorphism is computed in \cite{ASS12}. Thus $\WAut(\A)=\langle\tau, \sigma\rangle\cong D_5$.
On the other hand, $\Aut(\A)=\Aut_0(\A)=\{id\}$.
\end{example}


\section{Cluster automorphism groups and quasi-automorphism groups}\label{sect 4}
In this section, we consider the relations between the cluster automorphism groups and the quasi-automorphism groups. First, we reveal the relations between these at the matrix level.

\begin{lemma}\label{lem:relations-3-homo}
Let $\A$ be a cluster algebra of a cluster pattern $t \mapsto \S_t$ over a coefficient semifield $\P$. Let $\Psi: \A \rightarrow \A$ be a quasi-automorphism with $\Psi(\S_t) \asymp \S_{t'}$. Let $M_{\Psi;t}$ be the matrix in Lemma \ref{lemma:matrix-describe-QHom} satisfying
$B_{t'}= M_{\Psi;t}{B_t}$. Then
\begin{equation}\label{eq:relations-3-homo}
M_{\Psi;t}=\left(
         \begin{array}{cc}
           I & 0 \\
           M_{1;t} & M_{2;t} \\
         \end{array}
       \right)
,
\end{equation}
where $I$ is the identity matrix of rank $n$, $M_{1;t}$ is a $(m-n)\times n$ integer matrix, and $M_{2;t}$ is a $(m-n)\times (m-n)$ integer matrix. Conversely, each matrix $M$ of this form defines a quasi-automorphism.

Moreover

{\rm (1)} $\Psi$ is a cluster automorphism if and only if $M_{1;t}=0$ and $M_{2;t}$ is a permutation matrix for some (and thus every) $t\in \T_n$;

{\rm (2)} $\Psi$ is a weak cluster automorphism if and only if $M_{1;t}=0$ and $B_{2,t'}=M_{2;t}B_{2,t}$ for every $t\in \T_n$, where  $B_{t}=\left(
                       \begin{smallmatrix}
                         \underline{B}_t \\
                         B_{2,t}\\
                       \end{smallmatrix}
                     \right)$
and $B_{t'}=\left(
                       \begin{smallmatrix}
                         \underline{B}_{t'} \\
                         B_{2,t'}\\
                       \end{smallmatrix}
                     \right)$.

In particular, there are injective maps given by inclusions of sets
 \[\Aut^+(\A)\monoto \WAut^+(\A)\monoto \QAut \cala.\]
\end{lemma}
\begin{proof} Recall that $\Psi(x_{j,t})=\prod_{i=1}^{ m}x_{i,t'}^{m_{ij}}$, for all $1\le j \le m$.
Since $\Psi$ is a quasi-homomorphism, $\Psi(x_{j;t}) \asymp x_{j;t'}$ for all $1 \leq j\leq n$, that is,
\begin{equation}
\label{eq  4.1} \frac{\Psi(x_{j;t})}{x_{j;t'}}=\frac{\prod_{i=1}^{ m}x_{i,t'}^{m_{ij}}}{x_{j;t'}}\in \P.
\end{equation}
 Thus the top left block of the matrix $M$ is the identity matrix and the block $M_{1;t}$ determines the quotient in (\ref{eq 4.1}).
On the other hand, if $n+1\le j \le m$, we have $\Psi(x_{j,t})=\prod_{i=1}^{ m}x_{i,t'}^{m_{ij}}\in \P$, since $\Psi$ maps coefficient semifield to itself, whence the top right block is the zero matrix.

(1) Note that $M_{1;t}=0$ if and only if $\Psi(\x_t)$ is a cluster, $M_{2;t}$ is a permutation matrix if and only if $\Psi(\p)=\p$ as unordered sets. Moreover, under the above two conditions, $\Psi(\hat{y}_{i;t}) = \hat{y}_{i;t'}$ is equivalent to the equality \eqref{eq:cluster-automorphisms}.
Then comparing the definition of quasi-homomorphisms and the definition of cluster automorphisms, we are done.

(2) A weak cluster automorphism maps cluster variables to cluster variables; this is equivalent to $M_{1;t}=0$ and $B_{2,t'}=M_{2;t}B_{2,t}$ for every $t\in \T_n$.
\end{proof}

\begin{remark}
The matrix $M_{2;t}$ is determined by the action of $\Psi$ on the coefficients semifield $\P$ and therefore it is independent of  $t\in \T_n$. Thus the cases (1) and (2) of the lemma imply that if $\Psi$ is a cluster automorphism or a weak cluster automorphism then  the matrix $M_{\Psi,t}$ does not depend on $t\in \T_n$. In these cases, we shall denote the matrix simply by $M_\Psi$. However, this is not true for the general quasi-homomorphism. For a quasi-automorphism $\Psi$, it seems interesting to consider the relations between $M_{\Psi,t}$ for different $t\in \T_n$. That is, consider the action of the seed mutations on the matrix $M_{\Psi,t}$.
\end{remark}

\begin{proposition}\label{prop:main-1}
Let $\A=\cala(B)$ be a cluster algebra.

{\rm (1)} There are injective group homomorphisms
\[\Aut^+_0(\A)\monoto\WAut_0^+(\A)\monoto \QAut_0(\A).\]

{\rm (2)} If $\A$ satisfies the $(\star)$ condition
then there are injective group homomorphisms
\[\Aut^+_0(\A)\monoto\WAut_0^+(\A)\monoto\Aut^+(\A_{triv}).\]
\end{proposition}
\begin{proof}
(1) This immediately follows from Lemma \ref{lem:relations-3-homo}.

(2)
The existence of the first homomorphism follows from part (1). To prove the existence of the second, let $\Psi\in\WAut^+(\A)$, $\zS=(\x,\p,B)$ be a seed in $\A$ and denote by $\zS'=(\x',\p,B')$ its image under $\Psi$.
Then $\underline{B}\cong\underline{B'}$. When we specialize all coefficients to 1, we see that $\Psi$ induces a direct cluster automorphism $\underline{\Psi} $ on $\A_{triv}$. Moreover, if $\Psi$ is a coefficient permutation then $\underline{\Psi}$ is the identity, and therefore the map $\Psi\mapsto \underline{\Psi}$ induces a group homomorphism
$\WAut_0^+\to\Aut^+(\A_{triv})$.

Now assume that $\Psi'\in\WAut^+(\A)$ is another weak automorphism such that $\underline{\Psi}=\underline{\Psi'}$. Then, obviously, $\underline{\Psi}$ and $\underline{\Psi'}$ induce the same automorphism on the exchange graph of $\A_{triv}$. But since $\A$ satisfies the ($\star$) condition, Lemma \ref{lem:dependence-exchange-graph} implies that $\A$ and $\A_{triv}$ have isomorphic exchange graphs. In particular, $\Psi-\Psi'$ fixes each cluster variable and hence $\Psi=\Psi'$  in $\WAut_0(\A)$. Thus the homomorphism $\Psi\mapsto \underline{\Psi}$ is injective.
\end{proof}

In the case where the cluster algebra has universal coefficients, we have the following result, which will be used in the proof of Theorem~\ref{thm:main-1}.

\begin{proposition}\label{prop:main-13}
Let $\A$ be a cluster algebra with universal coefficients. Then there are injective group homomorphisms
\[\Aut^+(\A_{triv})\monoto \WAut_0^+(\A)\monoto  \QAut_0(\A).\]
\end{proposition}
\begin{proof}
The existence of the second homomorphism follows from Proposition \ref{prop:main-1} (1). To show the existence of the first, let $f\in \Aut^+(\A_{triv})$, $\underline{\zS}=(\underline{\x},\underline{B})$ be a seed in $\cala$ and $\underline{\zS'}=(\underline{\x'},\underline{B'})$ its image under $f$. Then $\underline{B}\cong\underline{B'}$.
Choose two lifts $\zS=(\x,\p,B)$ and $\zS'=(\x',\p,B')$ of $\underline{\zS}$ and $\underline{\zS'}$ in $\A$. We consider a second copy of $\A$ which we denote by $\A'$ with the difference that  we use $\zS$ as the initial seed for $\A$ and $\zS'$ as the initial seed for $\A'$. Since  both $\A$ and $\A'$ have universal coefficients, there exist two coefficient specializations
$\Psi\colon\A\to\A'$ and $\Psi'\colon\A'\to\A$  such that $\Psi(\zS)=\zS'$ and $\Psi'(\zS')=\zS$. Moreover, the composition $\Psi'\circ\Psi$ is a coefficient specialization from $\A$ to itself that fixes $\zS$. By the uniqueness of the coefficient specialization for universal coefficients, we see that $\Psi'\circ\Psi$ is the identity on $\A$. Similarly, $\Psi\circ\Psi'$ is the identity too. In particular, $\Psi $ is a $\Z$-automorphism of $\A$. Moreover, $\Psi(\x)=\x'$ is a cluster and $\Psi(\p)\in \P$.
 Also $\Psi$ commutes with mutations, since it maps all clusters to clusters. Thus $\Psi $ is a weak cluster automorphism and the mapping $f\mapsto \Psi$ defines a group homomorphism $\Aut^+(\A_{triv})\to \WAut_0^+(\A)$. This homomorphism in injective with left inverse given by the map $\Psi\to \underline{\Psi}$ obtained by specializing all coefficients to 1.
\end{proof}

\begin{example}\label{ex WAutcont}
 One can show that
 the cluster algebra $\A$ in Example \ref{ex WAut} has universal coefficients, although this realization of universal coefficients is different from the one given in \cite{FZ07}.
We have seen  $\WAut(\A)=\WAut_0(\A)$ is isomorphic to the dihedral group $D_5$ and in this case it is also isomorphic to $\Aut(\A_{triv})$, see \cite{ASS12}.
%
\end{example}

Now we consider the relations between the quasi-automorphism group and the cluster automorphism group.
\begin{proposition}\label{prop:main-2}
Let $\A$ be a cluster algebra that satisfies the $(\star)$ condition of Lemma \ref{lem:dependence-exchange-graph}, then we have injective group homomorphisms
\[\Aut_0^+(\A)\monoto \QAut_0(\A)\monoto \Aut^+(\A_{triv}).\] Moreover, if $\A$ is gluing free, then $\Aut^+(\A)\monoto \QAut_0(\A)$.
\end{proposition}
\begin{proof}

We start by proving the existence of the first homomorphism. By Lemma \ref{lem:relations-3-homo},
the inclusion $h\colon\Aut^+(\A)\monoto\QAut(\A)$ is an injective map of sets. Since coefficient permutations map each seed $\zS_t$ to itself, and therefore are proportional to the identity on $\A$, the map $h$ induces a homomorphism of groups $h_0\colon\Aut^+_0(\A)\to\QAut_0(\A)$. To show that $h_0$ is injective, suppose we have $\Psi\in\Aut_0^+(\A)$ such that $h_0(\Psi)=1_{\QAut_0(\A)}$. Thus $h_0(\Psi)$ is proportional to the identity, which implies $\Psi(x_{i;t})=p_{i;t}x_{i;t}$ with $p_{i;t}\in\P$,
for any $1 \leq i\leq n$ and $t\in \T_n$. Since $\Psi\in \Aut^+(\A)$, we see that $(p_{1,t}x_{1;t},\cdots,p_{n,t}x_{n;t})$ is a cluster of $\A$.
By our assumption, the $(\star)$ condition is satisfied, hence  Lemma \ref{lem:proportional seed which are equal} implies that the two clusters are the same, and thus  $p_{i,t}=1$ for all $i,t$. Therefore $\Psi$ preserves the cluster variables and thus $\Psi$ is the identity in $\QAut_0(\A)$.

To show the existence of the second homomorphism, let $\Psi\in\QAut(\A)$ and $M_\Psi=\left(\begin{smallmatrix}I&0\\M_1&M_2\end{smallmatrix}\right)$ be the matrix described in Lemma~\ref{lem:relations-3-homo}. Specializing all coefficients to 1, we obtain a quasi-automorphism $\Psi_{triv}\in\QAut(\A_{triv})$ with matrix $M_{\Psi_{triv}}=I$. By Lemma~\ref{lem:relations-3-homo}, we also have $\Psi_{triv}\in\Aut^+(\A_{triv})$ and the rule $\Psi\mapsto\Psi_{triv}$ gives a map $h\colon \QAut(\A)\to\Aut^+(\A_{triv})$.
Moreover $\Psi$ is proportional to the identity on $\A$ if and only if $h(\Psi) $ is the identity on $\A_{triv}$. Thus $h$ induces an injective
homomorphism $\QAut_0(\A)\monoto \Aut^+(\A_{triv})$.

Finally, the last statement of the lemma follows from the observation that if $\A$ is gluing free then $\Aut^+_0(\A)\cong \Aut^+(\A)$.
\end{proof}

In general, the two injective homomorphisms of the previous proposition are not surjective. We give two examples below; we also refer the reader to Example 6.9 in \cite{F16}.

\begin{example}\label{exm:counter-exm1}
Let $\A$ be the cluster algebra of type $A_3$ given by the seed $\S=((x_1,x_2,x_3)$, $(x_4),B)$, where the matrix $B$ and its quiver $Q(B)$ are as follows
 \[B=\left(
       \begin{array}{ccc}
         0 & 1 & 0 \\
         -1 & 0 & -1 \\
         0 & 1 & 0 \\
         0 & 0 & 1 \\
       \end{array}
     \right)
\qquad \qquad
Q(B)= \vcenter{
\xymatrix{1\ar[r]&2&3\ar[l]&\framebox{4}~.\ar[l]}}
\]
After a sequence of mutations $\mu_3\mu_1\mu_2$ on $\S$ (the mutation order is $\mu_2$, $\mu_1$ and then $\mu_3$), we obtain a new seed $\S'=((x'_1,x'_2,x'_3),(x_4),B')$, where
 \[B'=\left(
       \begin{array}{ccc}
         0 & 1 & 0 \\
         -1 & 0 & -1 \\
         0 & 1 & 0 \\
         0 & 0& -1 \\
       \end{array}
     \right)
     \qquad\qquad
Q(B')= \vcenter{
\xymatrix{1\ar[r]&2&3\ar[l]\ar[r]&\framebox{4}~.}}
\]
Since $B\ncong B'$ (or equivalently $Q(B)\ncong Q(B')$), there is no direct cluster automorphism on $\A$ which maps $\S$ to $\S'$. However, since $\underline{B}=\underline{B}'$ and $\Lat(B)=\Lat(B')$, there exists a quasi-automorphism which maps $\S$ to a seed proportional to $\S'$. In fact the map
\begin{equation*}
\label{map:f-1}
f:
\begin{cases}
x_{i} \mapsto x'_{i} & \text{if $i = 1,2,3$}; \\ 
x_4 \mapsto x^{-1}_4
\end{cases}
\end{equation*}
induces such a quasi-automorphism. Thus \[\Aut_0^+(\A)\ncong \QAut_0(\A).\]
\end{example}

\begin{example}\label{exm:counter-exm2}
 Let $\A$ be a cluster algebra of $D_4$ type with a seed $\S=((x_1,x_2,x_3,x_4),(x_5),B)$, where
 \[B=\left(
       \begin{array}{cccc}
         0 & 1 & 1 & -1 \\
         -1 & 0 & 0 & 0 \\
         -1 & 0 & 0 & 0 \\
         1 & 0 & 0 & 0 \\
         0 & 1 & 0 & 0 \\
       \end{array}
     \right).
 \qquad\qquad
Q(B)= \vcenter{
\xymatrix{\framebox{5}\ar[r]&2&&\\
&&1\ar[ul]\ar[dl]&4~.\ar[l]\\
&3&&}}
\]

A permutation of the cluster variables $x_2$ and $x_3$, produces another  seed $\S'=((x_1,x_3,x_2,x_4),(x_5),B')$, where
 \[B'=\left(
       \begin{array}{cccc}
         0 & 1 & 1 & -1 \\
         -1 & 0 & 0 & 0 \\
         -1 & 0 & 0 & 0 \\
         1 & 0 & 0 & 0 \\
         0 & 0 & 1 & 0 \\
       \end{array}
     \right)
\qquad\qquad
Q(B')= \vcenter{
\xymatrix{\framebox{5}\ar[r]&3&&\\
&&1\ar[ul]\ar[dl]&4~.\ar[l]\\
&2&&}}
\]

The last row of $B'$ is not contained in the integer row span of $B$, thus $\Lat(B)\neq \Lat(B')$, and there is no quasi-automorphism that maps $\S$ to a seed proportional to $\S'$. However, changing cluster variables $x_2$ and $x_3$ induces a direct cluster automorphism on $\A_{triv}$ which maps $\S$ to $\S'$.
Thus \[\QAut_0(\A)\ncong \Aut^+(\A_{triv}).\]
\end{example}

\subsection{Main result}
In view of Example \ref{exm:counter-exm2} it is  natural to ask under which conditions we do have an isomorphism between the groups $\QAut_0(\A)$ and $ \Aut^+(\A_{triv})$.
The following theorem gives three such conditions, showing that the groups are isomorphic for a large class of cluster algebras.

\begin{theorem}\label{thm:main-1}
Let $\A$ be a cluster algebra. There is an isomorphism
\[\QAut_0(\A)\cong \Aut^+(\A_{triv})\]
in each of the following cases.

{\rm (1)} $\A$ has principal coefficients and satisfies the $(\star)$ condition of Lemma \ref{lem:dependence-exchange-graph}.

{\rm (2)} $\A$ has universal coefficients.

{\rm (3)} $\A$ is a surface cluster algebra with boundary coefficients.
\end{theorem}
\begin{proof} (1) We will show that the injective homomorphism $h_0\colon \QAut_0(\A)\monoto \Aut^+(\A_{triv})$  from Proposition \ref{prop:main-2} is surjective.
Let $\Psi\in\Aut^+(\A_{triv})$, let $\zS_{triv}$ be a seed in $\A_{triv}$ and let $\zS'_{triv}=\Psi(\zS_{triv})$ be its image under $\Psi$. Also denote by $x_j'=\Psi(x_j)$ the image of the cluster variable $x_j\in\zS$ under $\Psi$.
Since $\A$ satisfies the condition $(\star)$, Lemma \ref{lem:dependence-exchange-graph} implies that there are unique seeds $\zS$ and $\zS'$ in $\A$ that specialize to $\zS_{triv}$ and $\zS'_{triv}$, respectively, when coefficients are sent to 1.
Let $B=\left(
                       \begin{smallmatrix}
                         \underline{B} \\
                         G \\
                       \end{smallmatrix}
                     \right)$
and  $B'=\left(
                       \begin{smallmatrix}
                         \underline{B}' \\
                         G' \\
                       \end{smallmatrix}
                     \right)$
be the matrices of $\S$ and $\S'$ respectively. It was shown in \cite[Theorem 1.2]{NZ12} that, since $\A$ has principal coefficients, the matrices $G$ and $G'$ are both integer matrices with determinate $\pm 1$. In particular, they are invertible and we can define the following matrix \[M=\left(\begin{array}{cc} I&0\\0&G'G^{-1}\end{array}\right).\]
According to Lemma \ref{lem:relations-3-homo}, the matrix $M$ defines a unique quasi-automorphism $\widetilde\Psi$ mapping $\zS$ to $\zS'$ by
 \[\widetilde\Psi(x_j)=\prod_{i=1}^{2n}(x_i')^{m_{ij}} =x_j'\prod_{i=n+1}^{2n}(x_i')^{m_{ij}}.\]
 Now, since the homomorphism $h_0$ maps frozen variables to 1, we have $h_0(\widetilde{\Psi})(x_j) =x_j'=\Psi(x_j)$. Hence $\Psi$ is in the image of $h_0$ and thus $h_0$ is surjective.

(2) According to Proposition \ref{prop:main-13} and  Proposition \ref{prop:main-2} there are injective homomorphisms $\Aut^+(\A_{triv})\monoto \QAut_0(\A)$ and $\QAut_0(\A)\monoto \Aut^+(\A_{triv})$, both of which are induced by inclusions of sets.
Thus $\QAut_0(\A)\cong \Aut^+(\A_{triv})$.

(3) Cluster algebras from surfaces are skew-symmetric. Moreover, using our assumption that the rank of a cluster algebra is at least two, it follows from \cite[Proposition~3.10]{CZ16} that the surface cluster algebra is gluing free. Therefore we can apply Proposition~\ref{prop:main-2} which yields  injective homomorphisms
 $\Aut_0^+(\A)\monoto \QAut_0(\A)\monoto \Aut^+(\A_{triv})$. On the other hand, Theorem 3.18 in \cite{CZ16} shows that $\Aut_0^+(\A)\cong \Aut^+(\A_{triv})$. Therefore $\QAut_0(\A)\cong \Aut^+(\A_{triv})$.
\end{proof}
\begin{remark}
 Part (3) of the theorem has an interpretation in terms of the tagged mapping class  group  $\mathcal{MG}_{\bowtie}(S,M)$ of the surface $(S,M)$ introduced in \cite{ASS12}. For all surfaces with the exceptions of a sphere with four punctures, a once-punctured square, or a digon with one or two punctures, we have $\Aut^+(\A_{triv})\cong\mathcal{MG}_{\bowtie}(S,M)$ (this was conjectured in \cite{ASS12} and proved in \cite{Bridgeland-Smith}). Thus we see that, unless $(S,M)$ is one of these four exceptions, the quasi-automorphism group of the cluster algebra of $(S,M)$ with boundary coefficients is isomorphic to the tagged mapping class group. This result recovers a very special case of  Fraser's Theorem \cite[Theorem 7.5]{F16}.
\end{remark}
\section{Finite type}\label{sect 5}
In this section we will show that $\QAut_0(\A)\cong \Aut^+(\A_{triv})$ for finite type cluster algebra $\A$ with arbitrary coefficients, except for type $\mathbb{D}_{n}$ with $n $ even.

Let $\A$ be a cluster algebra of finite type and let $\S=(\x,\p,B)$ be a seed in $\A$ such that the valued quiver $\underline{Q}$ of $\underline{B}$ is bipartite.
Denote  the sources of the quiver $\underline{Q}$ by $i_1,\cdots,i_s$, and its sinks  by  $i_{s+1},\cdots,i_n$. Following \cite{FZ03a}, we define $\tau_+=\mu_{i_s}\cdots \mu_{i_1}$ and $\tau_-=\mu_{i_n}\cdots \mu_{i_{s+1}}$ the compositions of seed mutations at sources and sinks respectively. Since mutations at any two sources (sinks)  commute with each other, the definition of $\tau_{\pm}$ is independent of the order of the mutations. Then the group generated by $\tau_\pm$ is a dihedral group which is considered in \cite{FZ03a} to prove the periodicity conjecture of $Y$-systems. This group is also closely related to the cluster automorphism groups, see for example \cite{ASS12,CZ16b}. We define $\tau=\tau_+\tau_-$. It corresponds to the Auslander-Reiten translation on the corresponding cluster category
\begin{lemma}\label{lem:finite-type}
Let $\A$ be a cluster algebra of finite type with arbitrary coefficients and assume that $\A$ is not of type $\mathbb{D}_{2\ell}$. Let $\S=(\x,\p,B)$ and $\S'=(\x',\p,B')$ be bipartite seeds of $\A$ such that  $\underline{B}'=\underline{B}$. Then $\Lat(B')=\Lat(B)$.
\end{lemma}
\begin{proof}
By Corollaries 3.3(1) and 3.6(2) of \cite{CZ16b},
there exists $m\in \mathbb{Z}$ such that $\tau^m\zS=\zS'$. So it suffices to show the statement for $\zS'=\tau^{-1}\zS$.
 According to Lemma \ref{lem one frozen row}, it is sufficient to consider the case where the matrix $B$ has only one frozen row. Denote by $\alpha_1,\cdots,\alpha_n, \beta$ the rows of $B$, and denote by $\alpha_1,\cdots,\alpha_n, \beta'$ the rows of $B'$.  Let $\beta=(b_1,\cdots,b_n)$ and let $\beta'=(b'_1,\cdots,b'_n)$.  Then we have to show $\beta'\in \Lat(B)$ and $\beta\in \Lat(B')$.

 Without loss of generality, we assume the vertex $1$ is a source of $\underline{Q}$.
 We will distinguish several cases.

{\bf{Case I:}} Type $\mathbb{A}_n$. Suppose first that $n$ is even. We have
\[\underline{B}=\left(
      \begin{array}{ccccccc}
        0  & 1 & 0    &      &      &   &  \\
        -1 & 0 & -1   & 0     &      &   &  \\
         0  & 1 &  0& 1 &  \ddots    &   &  \\
           &  \ddots & \ddots & \ddots & \ddots & \ddots  &  \\
           &   &     \ddots & \ddots & \ddots & -1 & 0 \\
           &   &      &     \ddots & 1   & 0 & 1 \\
           &   &      &      &     0 & -1 & 0 \\
      \end{array}
    \right)
.\]
Let $e_i=(0,\ldots,0,1,0,\ldots,0)$ the $i$-th standard basis vector of $\mathbb{Z}^n$. Then
\[ e_i=\left\{
\begin{array}{cc}
-\za_{i+1}+\za_{i+3}-\za_{i+5}\cdots\pm\za_n &\textup{if $i$ is odd;}\\
-\za_{i-1}+\za_{i-3}-\za_{i-5}\cdots\pm\za_1 &\textup{if $i$ is even.}
\end{array}\right.\]
Therefore $\Lat(\underline{B})=\Z^n$. So $\Lat(B)=\Z^n=\Lat(B')$ and the result is true.

 Now suppose that $n$ is odd, say $n=2\ell+1\ge 3.$ Then
$$\underline{B}=\left(
      \begin{array}{ccccccc}
        0  & 1 & 0    &      &      &   &  \\
        -1 & 0 & -1   &      &      &   &  \\
           & 1 &  \ddots& \ddots &      &   &  \\
           &   & \ddots & \ddots & \ddots &   &  \\
           &   &      & \ddots & \ddots & 1 &  \\
           &   &      &      & -1   & 0 & -1 \\
           &   &      &      &      & 1 & 0 \\
      \end{array}
    \right)
.$$
A direct calculation shows that
\begin{equation*}\label{eq:type-A-1}
b'_{i}  =
\begin{cases}
-b_1+[b_2+[b_1]_{+}+[b_3]_{+}]_{+} & \text{$i = 1$;} \\ 
-b_i-[b_{i-1}]_{+}-[b_{i+1}]_{+} & \text{$i$ even;} \\
-b_i+[b_{i+1}+[b_i]_{+}+[b_{i+2}]_{+}]_{+} & \\
~~~~~+[b_{i-1}+[b_{i-2}]_{+}+[b_{i}]_{+}]_{+} & \text{$i$ odd and $i\neq 1, n$;} \\
-b_n+[b_{n-1}+[b_{n-2}]_{+}+[b_n]_{+}]_{+} & \text{$i = n$.}
\end{cases}
\end{equation*}
So \begin{equation*}
\label{eq:type-A-2}
\beta' = -\beta-\sum_{k=0}^\ell[b_{2k+1}]_{+}\alpha_{2k+1}-\sum_{k=1}^\ell[b_{2k}+[b_{2k-1}]_{+}+[b_{2k+1}]_{+}]_{+}\alpha_{2k}.
\end{equation*}
Therefore $\beta'\in \Lat(B)$ and $\beta\in \Lat(B')$.

{\bf{Case II:}} Type $\mathbb{B}_n$. In this case we assume that $\beta=(\frac{b_1}{2},\cdots,b_n)$ and $\beta'=(b'_1,\cdots,b'_n)$. If $n$ is odd, say $n=2\ell+1\ge 3$,  we have
$$\underline{B}=\left(
      \begin{array}{ccccccc}
        0  & 2 & 0    &      &      &   &  \\
        -1 & 0 & -1   &      &      &   &  \\
           & 1 & \ddots & \ddots &      &   &  \\
           &   & \ddots & \ddots & \ddots &   &  \\
           &   &      & \ddots & \ddots & 1 &  \\
           &   &      &      & -1   & 0 & -1 \\
           &   &      &      &      & 1 & 0 \\
      \end{array}
    \right)
.$$
Then
\begin{equation*}
\label{eq:type-B-1}
b'_{i}  =
\begin{cases}
-\frac{1}{2}b_1+[b_2+[b_1]_{+}+[b_3]_{+}]_{+} & \text{$i = 1$;} \\ 
-b_i-[b_{i-1}]_{+}-[b_{i+1}]_{+} & \text{$i$ even;} \\
-b_i+[b_{i+1}+[b_i]_{+}+[b_{i+2}]_{+}]_{+} & \\
~~~~~~+[b_{i-1}+[b_{i-2}]_{+}+[b_{i}]_{+}]_{+} & \text{$i$ odd and $i\neq 1, n$;} \\
-b_n+[b_{n-1}+[b_{n-2}]_{+}+[b_n]_{+}]_{+} & \text{$i = n$.}
\end{cases}
\end{equation*}
So
\begin{equation*}
\label{eq:type-B-2}
\beta' = -\beta-\frac{[b_1]_+}{2}\alpha_1-\left(\sum_{k=1}^\ell[b_{2k+1}]_{+}\alpha_{2k+1}+[b_{2k}+[b_{2k-1}]_{+}+[b_{2k+1}]_{+}]_{+}\alpha_{2k}\right).
\end{equation*}

If $n$ is even, say $n=2\ell$, we have
$$\underline{B}=\left(
      \begin{array}{ccccccc}
        0  & 2 & 0    &      &      &   &  \\
        -1 & 0 & -1   &      &      &   &  \\
           & 1 & \ddots & \ddots &      &   &  \\
           &   & \ddots & \ddots & \ddots &   &  \\
           &   &      & \ddots & \ddots & -1 &  \\
           &   &      &      & 1   & 0 & 1 \\
           &   &      &      &      & -1 & 0 \\
      \end{array}
    \right)
.$$
Then
\begin{equation*}
\label{eq:type-B-3}
b'_{i}  =
\begin{cases}
-\frac{1}{2}b_1+[b_2+[b_1]_{+}+[b_3]_{+}]_{+} & \text{$i = 1$;} \\ 
-b_i-[b_{i-1}]_{+}-[b_{i+1}]_{+} & \text{$i$ even;} \\
-b_i+[b_{i+1}+[b_i]_{+}+[b_{i+2}]_{+}]_{+} & \\
~~~~~~+[b_{i-1}+[b_{i-2}]_{+}+[b_{i}]_{+}]_{+} & \text{$i$ odd and $i\neq 1, n$;} \\
-b_n+[b_{n-1}]_+ & \text{$i = n$.}
\end{cases}
\end{equation*}
So \begin{equation*}
\label{eq:type-B-4}
\beta' = -\beta-\frac{[b_1]_+}{2}\alpha_1-\left(\sum_{k=1}^{\ell-1}[b_{2k+1}]_{+}\alpha_{2k+1}+[b_{2k}+[b_{2k-1}]_{+}+[b_{2k+1}]_{+}]_{+}\alpha_{2k}\right).
\end{equation*}

{\bf{Case II:}} Type $\mathbb{C}_n$. This case is similar to the proof of type $\mathbb{B}_n$.

{\bf{Case III:}} Type $D_n, n={2\ell+1}, \ell\geqslant 2$. We consider the exchange matrix
$$\underline{B}=\left(
      \begin{array}{ccccccccc}
        0  & 1 &  1  & 1 &      &      &      &   &  \\
        -1 & 0 &  0  & 0 &      &      &      &   &  \\
        -1 & 0 &  0  & 0 &      &      &      &   &  \\
        -1 & 0 &  0  & 0 & -1   &      &      &   &  \\
           &   &     & 1 & \ddots & \ddots &      &   &  \\
           &   &     &   & \ddots & \ddots & \ddots &   &  \\
           &   &     &   &      & \ddots & \ddots & 1 &  \\
           &   &     &   &      &      & -1   & 0 & -1 \\
           &   &     &   &      &      &      & 1 & 0 \\
      \end{array}
    \right)
.$$
Then we have
\begin{equation*}
\label{eq:type-D-1}
b'_{i}  =
\begin{cases}
-b_1+[[b_2]_++[b_1]_{+}]_{+}+[[b_3]_++[b_1]_{+}]_{+} &  \\ 
~~~~~~+[b_4+[b_1]_++[b_5]_{+}]_{+} & \text{$i = 1$;} \\
-b_i-[b_1]_+                       & \text{$i = 2,3$;} \\
-b_i-[b_{i-1}]_{+}-[b_{i+1}]_{+}   & \text{$i$ even and $i\geqslant 4$;} \\
-b_i+[b_{i+1}+[b_i]_{+}+[b_{i+2}]_{+}]_{+} & \\
~~~~~~+[b_{i-1}+[b_{i-2}]_{+}+[b_{i}]_{+}]_{+} & \text{$i$ odd and $5\leqslant i \leqslant n-1$;} \\
-b_n+[b_{n-1}+[b_{n-2}]_{+}+[b_{n}]_{+}]_{+} & \text{$i = n$.}
\end{cases}
\end{equation*}
So \begin{equation*}
\label{eq:type-D-2}
\begin{aligned}
\beta' =& -\beta-[b_1]_+\alpha_1-[b_2+[b_1]_+]_+\alpha_2-[b_3+[b_1]_+]_+\alpha_3\\
& -\left(\sum_{k=2}^\ell[b_{2k+1}]_{+}\alpha_{2k+1}+\sum_{k=2}^\ell[b_{2k}+[b_{2k-1}]_{+}+[b_{2k+1}]_{+}]_{+}\alpha_{2k}\right).
\end{aligned}
\end{equation*}

{\bf{Case IV:}} Type $E_7$. We consider the exchange matrix
$$\underline{B}=\left(
      \begin{array}{ccccccc}
        0  & 1 & 1    &  1    &  0    & 0  & 0 \\
        -1 & 0 & 0   &  0    &   -1   &  0 & 0 \\
         -1  & 0 & 0 & 0 &  0    & 0  & 0 \\
          -1 &  0 & 0 & 0 & 0 & -1  & 0 \\
          0 & 1  &  0    & 0 & 0 & 0 & 0 \\
          0 &  0 &    0  &    1  & 0   & 0 & 1 \\
           0 & 0  &  0    &  0    &  0    & -1 & 0 \\
      \end{array}
    \right)
.$$
Then we have
\begin{equation*}
\label{eq:type-D-1a}
\begin{cases}
b_1= -b_1+[b_2+[b_1]_+[b_5]_+]_++[b_3+[b_1]_+]_+ +[b_4+[b_1]_++[b_6]_+]_+\\
b_2=-(b_2+[b_1]_++[b_5]_+) \\
b_3=-(b_3+[b_1]_+) \\
b_4=-(b_4+[b_1]_++[b_6]_+) \\
b_5=-b_5+[b_2+[b_1]_++[b_5]_+]_+ \\
b_6=-b_6+[b_4+[b_1]_++[b_6]_+]_++[b_7+[b_6]_+]_+ \\
b_7=-(b_7+[b_6]_+)
\end{cases}
\end{equation*}
So \begin{equation*}
\label{eq:type-D-2a}
\begin{aligned}
\beta' =& -\beta-[b_1]_+\alpha_1-[b_2+[b_1]_++[b_5]_+]_+\alpha_2-[b_3+[b_1]_+]_+\alpha_3\\
& -[b_4+[b_1]_++[b_6]_+]_+\alpha_4-[b_6]_+\alpha_6-[b_7+[b_6]_+]_+\alpha_7.
\end{aligned}
\end{equation*}

{\bf{Case V:}} For types $E_6$, $E_8$ and $F_4$, the lattice $\Lat(\underline{B})$ is $\Z^6$, $\Z^8$ and $\Z^4$ respectively, so the result is true by the same argument as for type $\mathbb{A}_n$ with $n $ even.

{\bf{Case VI:}} The type $G_2$ is a special case of Theorem \ref{thm:main-3}.
\end{proof}

\smallskip
We are now ready for the main result of this section.
\begin{theorem}\label{thm:main-2} Let $\A$ be a cluster algebra of finite type
 with arbitrary coefficients and assume that $\A$ is not of type $\mathbb{D}_{2\ell}$.
Then \[\QAut_0(\A)\ \cong\  \Aut^+(\A_{triv}).\]

Moreover, if $\A$ is of type $\mathbb{D}_{2\ell}$ then
\[G\rtimes \QAut_0(\A) \ \cong\  \Aut^+(\A_{triv}),\]
where $G$ is a subgroup of the symmetric group $S_3$, in type $\mathbb{D}_4$, and $G$ is a subgroup of $\mathbb{Z}_2=\mathbb{Z}/2\mathbb{Z}$, in types  $\mathbb{D}_{2\ell}$ with $2\ell\ge 6$.
\end{theorem}
\begin{proof} For the type $\mathbb{A}_1$ a direct computation shows that $\QAut(\A)\cong \mathbb{Z}_2\cong\Aut^+(\A)$. Assume now that $n>1$.
Proposition \ref{prop:main-2} yields an injective homomorphism $\QAut_0(\A)\monoto  \Aut^+(\A_{triv})$. By  Lemma \ref{lem:finite-type} and Remark \ref{rem:quasi-automorphism} this is an isomorphism unless $\A$ is of type $\mathbb{D}_{2\ell}$.

In type $\mathbb{D}_{2\ell}$  a similar calculation shows that the Auslander-Reiten translation $\tau$ can again be lifted as a quasi-automorphism of $\A$. However, there may be additional automorphisms of $\A_{triv}$ which are not induced by some power of $\tau^{m}$, as we have seen in Example \ref{exm:counter-exm2}. So we have $\langle\tau\rangle\subseteq \QAut_0(\A)\subseteq \Aut(\A_{triv})$.
By Table 1 in \cite{ASS12}, we know
\[\Aut(\A_{triv})\cong \left\{\begin{array}{ll} \mathbb{Z}_4\times S_3 &\textup{in type $\mathbb{D}_4$};\\
\mathbb{Z}_n\times \mathbb{Z}_2 &\textup{in types $\mathbb{D}_n$ with $n>4$},\end{array}\right.\]
 where the cyclic part is generated by the Auslander-Reiten translation $\tau$. This completes the proof.
\end{proof}
\begin{example}
For the cluster algebra of Example \ref{exm:counter-exm2} we have \[\QAut_0(\A)\cong\Z_4\times\Z_2\cong\langle\tau\rangle\times\langle\sigma\rangle,\] where $\tau$ is the Auslander-Reiten translation and $\sigma$ fixes $x_1$ and $x_3$ and interchanges $x_2$ and $x_4$. On the other hand, $\Aut^+(\A_{triv})\cong\Z_4\times S_3$, where the symmetric group $S_3$ is given by all permutations of the variables $x_2,x_3,x_4$. Since the symmetric group is a semidirect product of $\Z_3$ and $\Z_2$ we have $\Z_3\rtimes\QAut_0(\A)\cong\Aut^+(\A_{triv})$.
\end{example}

 As an immediate corollary, we obtain the complete list of all quasi automorphism groups of finite type cluster algebras.

\begin{corollary}\label{cor:finite-type-1}
Let $\A$ be a cluster algebra of finite type. Then the quasi-automorphism group $\QAut_0(\A)$ is listed in Table \ref{table}. Note that the group is a cyclic group, except for type $\mathbb{D}$. \end{corollary}
\begin{proof}
This follows from Theorem \ref{thm:main-2}, Table 1  in \cite{ASS12}.
\end{proof}

 We close this section with the following result on lattices which is a consequence of our arguments above.

\begin{corollary}\label{cor:finite-type-2} Let $\A$ be a cluster algebra of finite type and let $\zS, \zS'$
be two seeds such that the principal parts of the exchange matrices are equal.
Then the lattices of the exchange matrices coincide.
\end{corollary}
\begin{proof}
This follows from Remark \ref{rem:quasi-automorphism} and Theorem \ref{thm:main-2}.
\end{proof}

\section{Affine type}\label{section:affine type}
In this section we consider the quasi-automorphism groups for a cluster algebra $\A$ of affine type, that is, the cluster algebras of types $\widetilde{\mathbb{A}}_{p,q},(p\geqslant 1, q\geqslant 1)$, $\widetilde{\mathbb{D}}_{n-1},(n\geqslant 5)$, $\widetilde{\mathbb{E}}_{6}$, $\widetilde{\mathbb{E}}_{7}$ and $\widetilde{\mathbb{E}}_{8}$. We proceed case by case.

{\bf{Case I:}} Let $\A$ be a cluster algebra of type $\widetilde{\mathbb{A}}_{p,q}$. Let $\S=(\widetilde\x,\p,B)$
be a seed of $\A$, where the quiver of the exchange matrix $\underline{B}$ is as follows:

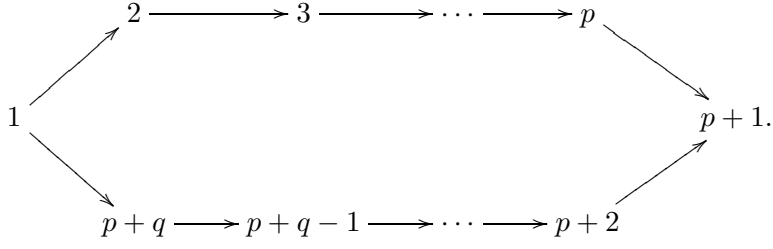
\[\xymatrix{&2\ar[r]&3\ar[r]&\cdots\ar[r]&p\ar[rd]\\
1\ar[ur]\ar[dr]&&&&&p+1 .\\
&p+q\ar[r]&p+q-1\ar[r]&\cdots\ar[r]&p+2\ar[ur]
}
\]
The cluster automorphism group of $\A_{triv}$ is computed in \cite{ASS12}.
Denote by $\S_{triv}=(\x=(x_1,\cdots,x_{p+q}),\underline{B})$ the initial seed of $\A_{triv}$, which corresponds to $\S$.
Let $\S'_{triv}=(\x'=(x'_1,\cdots,x'_{p+q}),\underline{B}')$ be a seed obtained from $\S_{triv}$ by an ordered sequence of mutations at $x_1$, $x_{p+q}$, $x_{p+q-1}$, $\cdots$ $x_{p+2}$. Let $\S''_{triv}=(\x''=(x''_1,\cdots,x''_{p+q}),\underline{B}'')$ be a seed obtained from $\S_{triv}$ by an ordered sequence of mutations at $x_1, x_{2}, \cdots x_{p}$. Then there are two cluster automorphisms $r_1$ and $r_2$ of $\A_{triv}$ which map $\S_{triv}$ to $\S'_{triv}$ and $\S''_{triv}$ respectively, where the actions on the initial cluster variables are as follows:
\begin{equation*}
\label{map:r-1}
r_1:
\begin{cases}
x_{i} \mapsto x'_{i+1} & \text{if $1\leqslant i\leqslant p+q-1$}; \\ 
x_{p+q} \mapsto x'_{1}.
\end{cases}
\end{equation*}

\begin{equation*}
\label{map:r-2}
r_2:
\begin{cases}
x_{1} \mapsto x''_{p+q}; \\ 
x_{i} \mapsto x''_{i-1} & \text{if $2\leqslant i\leqslant p+q.$}
\end{cases}
\end{equation*}
In fact if $p\neq q$, then
\begin{equation}\label{Hpq}\Aut^+(\A_{triv})\cong H_{p,q} = \langle  r_1, r_2| r_1r_2 = r_2 r_1, r_1^p=r_2^q \rangle,\end{equation} see  \cite[section 3.3]{ASS12}. If $p=q$, then $\Aut^+(\A_{triv})\cong H_{p,q}\rtimes \{1,\sigma\}$, where $\sigma$ is induced by the following permutation of the initial cluster variables:
\begin{equation*}
\label{map:sigma}
\sigma:
\begin{cases}
x_{1} \leftrightarrow x_{1}; \\ 
x_{p+1} \leftrightarrow x_{p+1};\\
x_{i} \leftrightarrow x_{p+q+2-i} & \text{if $2\leqslant i\leqslant p=q.$}
\end{cases}
\end{equation*}

Now we show the $r_1$ can be lifted as a quasi-automorphism of $\A$. We shall consider only the case where $p$ is even and $q$ is odd, since the other cases are similar. According to Lemma \ref{lem one frozen row}, we may also assume that $\A$ only has one coefficient. Let
$B=\left(
 \begin{smallmatrix}
 \underline{B} \\
 \beta \\
 \end{smallmatrix}
 \right)$
be the initial matrix of $\A$, where $\beta=(b_1,\ldots,b_{p+q})$.
Let
$B'=\left(
 \begin{smallmatrix}
 \underline{B} \\
 \beta' \\
 \end{smallmatrix}
 \right)$
be the matrix of a labeled cluster which is obtained from the labeled cluster $\widetilde\x'=\mu_{p+2}\cdots\mu_{p+q}\mu_1(\widetilde\x)$ after a relabeling $1\mapsto p+q\mapsto p+q-1\mapsto\cdots \mapsto 2\mapsto1$.
Denote the frozen row of $B'$ by $\beta'=(b'_1,\cdots,b'_{p+q})$. The cluster automorphism $r_1$ maps $\x$ to $\x'$. Note that after specializing the coefficient in $\widetilde\x'$ to $1$, we get $\x'$.
So to show the existence of a lift of $r_1$ in $\QAut_0(\A)$ which maps $\widetilde\x$ to $\widetilde\x'$
 After a direct calculation, we have
\begin{equation*}
\label{beta'-affine-type-A}
\begin{cases}

b'_{1}=b_{2}+[b_{1}]_+ & \\

b'_{i}=b_{i+1} & \text{if $2\leqslant i\leqslant p-1$}; \\

b'_{p}=\{b_{p+1}\} &  \\

b'_{p+1}=-\{b_{p+2}\} &  \\

b'_{p+i}=-\{b_{p+i+1}\}+[\{b_{p+i}\}]_+ & \text{if $2\leqslant i\leqslant q-1$}; \\

b'_{p+q}=-b_{1}+[\{b_{p+q}\}]_+ &
\end{cases}
\end{equation*}
where $\{b_{p+i}\}, 1\leqslant i\leqslant q$, is inductively defined by
\begin{equation*}
\label{beta'-affine-type-Aa}
\{b_{p+i}\}=\begin{cases}

b_{p+q}+[b_1]_+ &  \text{if $i=q$};\\

b_{p+i}+[\{b_{p+i+1}\}]_+ & \text{if $1\leqslant i\leqslant q-1$}.
\end{cases}
\end{equation*}
Finally, we have
$$\beta'=-\beta+a_1\alpha_1+\cdots+a_{p+q-2}\alpha_{p+q-2}+a_{p+q}\alpha_{a+q} ,$$
where the $\alpha$ are the rows of $\underline{B}$ and

\begin{equation*}
\label{map:linear coefficients}
\begin{cases}

a_{1}=b_{p+q}+b'_{p+q}; \\

a_{2k}=-\sum\limits_{i=1}^k (b_{2i-1}+b_{2i-1}')+\sum\limits_{i=1}^{p/2} (b_{2i}+b'_{2i})-\sum\limits_{i=1}^{(q-1)/2} (b_{p+2i}+b'_{p+2i})-b_{p+q}-b'_{p+q} \\
~~~~~~~~~~~~~~~~~~~~~~~~~~~~~~~~~~~~~~~~~~~~~~~~~~~~~~~~~~~~~~~~~~~~~~~~~~~~~~~~~~
~~~~~~~~~~~~~~~~~~~~~~~~~ \text{if $1\leqslant k\leqslant p/2$}; \\

a_{2k+1}=-\sum\limits_{i=1}^k (b_{2i}+b'_{2i})+b_{p+q}+b'_{p+q}
~~~~~~~~~~~~~~~~~~~~~~~~~~~~~~~~~~~~~~~~~~~~~~~~\text{if $1\leqslant k\leqslant p/2$};\\

a_{p+2k}=\sum\limits_{i=1}^k (b_{p+2i-1}+b'_{p+2i-1})-\sum\limits_{i=1}^{p/2} (b_{2i}+b'_{2i})
+\sum\limits_{i=1}^{(q-1)/2} (b_{p+2i}+b'_{p+2i})+b_{p+q}+b'_{p+q}\\
~~~~~~~~~~~~~~~~~~~~~~~~~~~~~~~~~~~~~~~~~~~~~~~~~~~~~~~~~~~~~~~~~~~~~~~~~~~~~~~~~~
~~~~~~~~~~~~~~~~~  \text{if $1\leqslant k\leqslant (q-1)/2$};\\

a_{p+2k+1}=\sum\limits_{i=1}^k (b_{p+2i}+b'_{p+2i})-\sum\limits_{i=1}^{p/2}(b_{2i}+b'_{2i})
+b_{p+q}+b'_{p+q}\\
~~~~~~~~~~~~~~~~~~~~~~~~~~~~~~~~~~~~~~~~~~~~~~~~~~~~~~~~~~~~~~~~~~~~~~~~~~~~~~~~~~~~~~~~~~~~~
~~~~~~\text{if $1\leqslant k\leqslant (q-1)/2$}.
\end{cases}
\end{equation*}
Similarly,  the cluster automorphism $r_2$ can also be lifted to $\QAut_0(\A)$.
 Therefore
 \[\begin{array}{rcccll}
 \QAut_0(\A)&\cong& H_{p,q}&\cong &\Aut^+(\A_{triv}) &\textup{if $p\neq q$;}\\
 \QAut_0(\A)&\cong& H_{p,q}\rtimes G&\subseteq &\Aut^+(\A_{triv}) &\textup{if $p=q$},
 \end{array}\]  where $G$ is a subgroup of $\Z_2$.

{\bf{Case II:}} Let $\A$ be a cluster algebra of type $\widetilde{\mathbb{D}}_{n-1}$. Let $\S=({\widetilde\x},\p,B)$ be a seed of $\A$, where the quiver of the exchange matrix $\underline{B}$ is

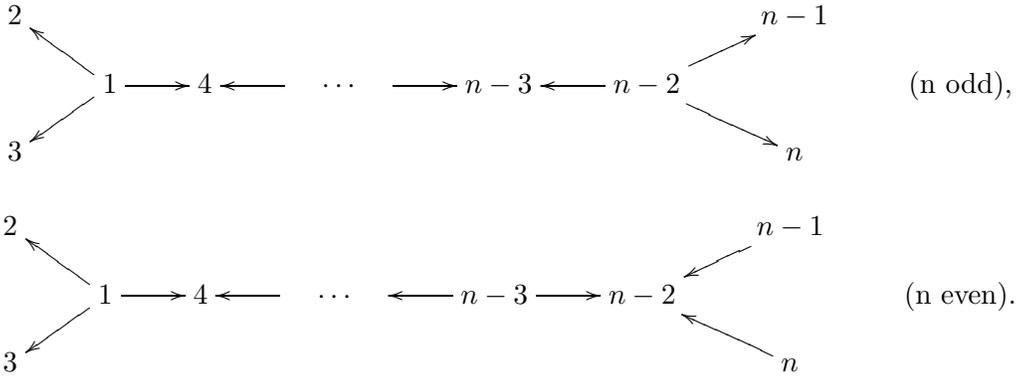
\begin{equation*}
\xymatrix@R10pt{
2&&&&&&n-1\\
&1\ar[lu]\ar[ld]\ar[r]&4&\ar[l]\quad\cdots\quad\ar[r]&n-3&\ar[l]n-2\ar[ru]\ar[rd]&&\text{(n odd),}\\
3&&&&&&n}
\end{equation*}

\begin{equation*}
\xymatrix@R10pt{
2&&&&&&n-1\ar[ld]\\
&1\ar[lu]\ar[ld]\ar[r]&4&\ar[l]\quad\cdots\quad&\ar[l]n-3\ar[r]&n-2&&\text{(n even).}\\
3&&&&&&n\ar[lu]}
\end{equation*}
It was shown in \cite{ASS12} that the direct cluster automorphism group $\Aut^+(\A_{triv})$ of $\A_{triv}$ is isomorphic to

\begin{equation}\label{G}
G=\left\langle \tau,\sigma,\rho_1,\rho_n \left| \begin{array}{c}
 \rho_i^2=1 ,\tau \rho_i=\rho_i \tau \ (i=1,n)\\
\tau\sigma=\sigma\tau,\ \sigma^2=\tau^{n-3} \\
\rho_1\sigma=\sigma\rho_n, \ \sigma\rho_1=\rho_n\sigma
\end{array}\right.\right\rangle.
\end{equation}
Here $\tau$ corresponds to the Auslander-Reiten translation of the cluster category. The automorphism $\sigma$ corresponds to a kind of reflection of the cluster category, and $\rho_1$ (respectively $\rho_2$) is induced by the automorphism of the initial quiver which permutes the vertices $2$ and $3$ (respectively $n-1$ and $n$) and preserves other vertices. We refer the reader to section 3.3 \cite{ASS12} for details.
 Similar to the situation in finite type $\mathbb{D}_{2\ell}$, the automorphisms
$\sigma$, $\rho_1$ and $\rho_2$ do not always lift to $\QAut_0(\A)$.
  However, we will see that $\tau$ can always be lifted.
Note that in this case the $\tau$ is still the composition $\tau_+\tau_-$, where $\tau_-$ (respectively $\tau_+$) is a sequence of mutations at the sinks (respectively sources).
Similar as the method we used previously, we consider any matrix
$B=\left(
 \begin{smallmatrix}
 \underline{B} \\
 \beta \\
 \end{smallmatrix}
 \right)$
and show that $\beta'$ is an integer linear combination of $\beta$ and the rows $\alpha_1, \cdots, \alpha_n$ of $\underline{B}$, where
$B'=\tau\left(
 \begin{smallmatrix}
 \underline{B} \\
 \beta' \\
 \end{smallmatrix}
 \right).$
By a direct calculation, we see this from the following equality
 \begin{equation*}
\label{eq:affine-type-D-1}
\begin{array}{l}
\beta' = -\beta-[-b_1+[-b_2]_++[-b_3]_++[-b_4]_+]_+\alpha_1-[-b_3]_+\alpha_3\\
~~~~~~~-\sum\limits_{k=2}^{(n-5)/2}[-b_{2k+1}+[-b_{2k}]_{+}+[-b_{2k+2}]_{+}]_{+}\alpha_{2k+1}
-\sum\limits_{k=1}^{(n-1)/2}[-b_{2k}]_{+}\alpha_{2k}\\
~~~~~~~-[-b_{n-2}+[-b_{n-3}]_++[-b_{n-1}]_++[-b_{n}]_+]_+\alpha_{n-2}
-[-b_n]_+\alpha_n.
\end{array}
\end{equation*}
for odd $n$ and the equality
 \begin{equation*}
\label{eq:affine-type-D-2}
\begin{array}{l}
\beta' = -\beta-[-b_1+[-b_2]_++[-b_3]_++[-b_4]_+]_+\alpha_1-[-b_3]_+\alpha_3\\
~~~~~~~-\sum\limits_{k=2}^{(n-4)/2}[-b_{2k+1}+[-b_{2k}]_{+}+[-b_{2k+2}]_{+}]_{+}\alpha_{2k+1}
-\sum\limits_{k=1}^{(n-2)/2}[-b_{2k}]_{+}\alpha_{2k}\\
~~~~~~~-[-b_{n-1}+[-b_{n-2}]_+]_+\alpha_{n-1}
-[-b_n+[-b_{n-2}]_+]_+\alpha_n.
\end{array}
\end{equation*}
for even $n$. Therefore $\QAut_0(\A)$ always contains a subgroup isomorphic to $\Z$ which is generated by $\tau$.

{\bf{Case III:}} Let $\A$ be a cluster algebra of type $\widetilde{\mathbb{E}}_{6}$. We begin with a initial seed of $\A$ with the exchange matrix $\underline{B}$, whose quiver is bipartite:

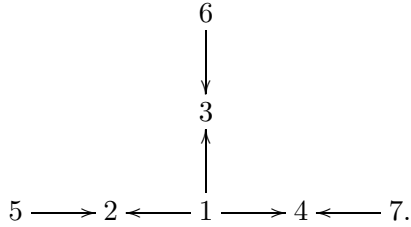
\[\xymatrix{
&&6\ar[d]&&\\
&&3&&\\
5\ar[r]&2&\ar[l]1\ar[r]\ar[u]&4&\ar[l]7 .}
\]
It was shown in \cite{ASS12} that the direct cluster automorphism group $\Aut^+(\A_{triv})$ is isomorphic to $\Z\times S_3$, where the cyclic part is generated by the Auslander-Reiten translation $\tau$ and the $S_3$ part comes from the automorphisms of the quiver. Again,  the cluster automorphisms in $S_3$ may not lift to $\QAut_0(\A)$, while $\tau$ can always be lifted. We still use the notations $\beta$, $\beta'$ and $\alpha_i, 1\leqslant i \leqslant 7$. The linear relation between them is
\begin{equation*}
\label{eq:type-tilde-E-6}
\begin{aligned}
\beta' =& -\beta-[-b_1+[-b_2]_++[-b_3]_++[-b_4]_+]_+\alpha_1-[-b_2]_+\alpha_2-[-b_3]_+\alpha_3\\
&-[-b_4]_+\alpha_4-[-b_5+[-b_2]_+]_+\alpha_5-[-b_6+[-b_3]_+]_+\alpha_6-[-b_7+[-b_4]_+]_+\alpha_7.
\end{aligned}
\end{equation*}
Therefore \[\QAut_0(\A) \cong \Z\times G\subseteq \Aut^+(\A_{triv}),\] where $G$ is a subgroup of $S_3$.

{\bf{Case IV:}} Let $\A$ be a cluster algebra of type $\widetilde{\mathbb{E}}_{7}$ with the initial principal part quiver
\[\xymatrix{
&&&3&&\\
6&\ar[l]5\ar[r]&2&\ar[l]1\ar[r]\ar[u]&4&\ar[l]7\ar[r]&8 .}
\]

 It was shown in \cite{ASS12} that the direct cluster automorphism group $\Aut^+(\A_{triv})$ is isomorphic to $\Z\times \Z_2$, where the infinite cyclic part is generated by the Auslander-Reiten translation $\tau$ and the $Z_2$ part comes from the automorphisms of the quiver.  Then \[\QAut_0(\A) \cong \Z\times G\subseteq \Aut^+(\A_{triv}),\] where $G$ is a subgroup of $\Z_2$. The linear relation is as follows:
\begin{equation*}
\label{eq:type-tilde-E-7}
\begin{aligned}
\beta' =& -\beta-[-b_1+[-b_2]_++[-b_3]_++[-b_4]_+]_+\alpha_1-[-b_2]_+\alpha_2-[-b_3]_+\alpha_3\\
&-[-b_4]_+\alpha_4-[-b_5+[-b_2]_++[-b_6]_+]_-\alpha_5-[-b_6]_+\alpha_6\\
&-[-b_7+[-b_4]_++[-b_8]_+]_+\alpha_7-[-b_8]_+\alpha_8.
\end{aligned}
\end{equation*}

{\bf{Case V:}} Let $\A$ be a cluster algebra of type $\widetilde{\mathbb{E}}_{8}$ with the initial principal part quiver
\[\xymatrix{
&&3&&\\
5\ar[r]&2&\ar[l]1\ar[r]\ar[u]&4&\ar[l]6\ar[r]&7&\ar[l]8\ar[r]&9.}
\]
The direct cluster automorphism group $\Aut^+(\A_{triv})=\langle\tau\rangle\cong \Z$. Then we have \[\QAut_0(\A) \cong \Z\cong \Aut^+(\A_{triv}).\] The linear relation is as follows:
\begin{equation*}
\label{eq:type-tilde-E-7a}
\begin{aligned}
\beta' =& -\beta-[-b_1+[-b_2]_++[-b_3]_++[-b_4]_+]_+\alpha_1-[-b_2]_+\alpha_2-[-b_3]_+\alpha_3\\
&-[-b_4]_+\alpha_4-[-b_5+[-b_2]_+]_+\alpha_5-[-b_6+[-b_4]_+-[-b_7]_+]_+\alpha_6\\
&-[-b_7]_+\alpha_7-[-b_8+[-b_7]_++[-b_9]_+]_+\alpha_8-[-b_9]_+\alpha_9.
\end{aligned}
\end{equation*}
To summarize, we have the following
\begin{theorem}\label{cor:affine-type}
Let $\A$ be a cluster algebra of affine type. Then the quasi-automorphism group $\QAut_0(\A)$ is listed in Table \ref{table}.
\end{theorem}

\begin{proof}
This follows from the above discussion.
\end{proof}

\section{Rank 2}\label{sect 7}
We end the paper with a discussion of rank two cluster algebras.
\begin{theorem}\label{thm:main-3}

Let $\A$ be a cluster algebra of rank two, then
\[\QAut_0(\A)\cong \Aut^+(\A_{triv})\cong\left\{\begin{array}{ll}
\Z_5 &\textup{in type $\mathbb{A}_2$;}\\
\Z_3 &\textup{in type $\mathbb{B}_2$;}\\
\Z_3 &\textup{in type $\mathbb{C}_2$;}\\
\Z_4 &\textup{in type $\mathbb{G}_2$;}\\
\Z &\textup{in all other rank 2 types.}
\end{array}\right.\]
\end{theorem}
\begin{proof}
Let $\zS=(\x,\p,B)$ be a seed. According to Lemma \ref{lem one frozen row}, we may assume without loss of generality that $B$ has exactly one frozen row. Then
\[B=\left(
 \begin{array}{cc}
0&s\\
 -t&0 \\
 b_1&b_2 \\
 \end{array}
 \right)
=\left(
 \begin{array}{c}
 \alpha_1 \\
 \alpha_2 \\
 \beta \\
 \end{array}
 \right)
\]
 where we may assume $s$ and $t$ are positive integers.  Let $\zS'=(\x',\p,B')$ be another seed and
 let  $\beta'=(b'_1,b'_2)$ denote the frozen row of $B'$. Then $\zS'= \tau^m(\zS)$ for some integer $m$. If $m=1$ then $\beta' =-\beta-[-b_1+t[-b_2]_+]_+\alpha_1-[-b_2]_+\alpha_2$. If $m=-1$ then $\beta' =-\beta-[b_1]_+\alpha_1-[b_2+s[b_1]_+]_+\alpha_2$. In both cases we have $\Lat(B)=\Lat(B')$, and by induction $\Lat(B)=\Lat(\tau^m(B))$ for any $m$. Therefore $\QAut_0(\A)\cong \Aut^+(\A_{triv})$.

If $\A$ is of finite type, that is types $\mathbb{A}_2$, $\mathbb{B}_2$,  $\mathbb{C}_2$, and $\mathbb{G}_2$, then the group has been computed in Corollary \ref{cor:finite-type-1}. Otherwise, the rank of $\tau$ is infinite, so $\QAut_0(\A)\cong \langle \tau\rangle\cong \Z$.
\end{proof}

\subsection*{Acknowledgments}

This joint work is done while the first author is visiting the University of Connecticut since December 2017. He gratefully acknowledges the support of the China Scholarship Council, and the University of Connecticut for providing him with an excellent working environment.

\end{document}